\documentclass[12pt,twoside]{amsart}

\usepackage[all]{xy}
\usepackage{mathrsfs,amsthm,amscd,amsmath,amssymb}
\usepackage[dvipdfmx]{hyperref}
\usepackage{color}
\usepackage{ulem}

\usepackage{color}

\title[Injectivity theorem for pseudo-effective 
line bundles]{Injectivity theorem  
for pseudo-effective line bundles  and its applications}

\date{\today, version 0.055}
\dedicatory{Dedicated to Professor Ichiro Enoki on the occasion of his retirement}

\author{Osamu Fujino}
\address{Department of Mathematics, Graduate School of Science, 
Kyoto University, Kyoto 606-8502, Japan}
\email{fujino@math.kyoto-u.ac.jp}

\author{Shin-ichi Matsumura}
\address{Mathematical Institute, Tohoku University, 
6-3, Aramaki Aza-Aoba, Aoba-ku, Sendai 980-8578, Japan.}
\email{mshinichi-math@tohoku.ac.jp, mshinichi0@gmail.com}

\newcommand{\Pic}[0]{\operatorname{Pic}}
\newcommand{\Coker}[0]{\operatorname{Coker}}
\newcommand{\Ker}[0]{\operatorname{Ker}}
\newcommand{\Image}[0]{\operatorname{Im}}

\newcommand{\deldel}{\sqrt{-1}\partial \overline{\partial}}
\newcommand{\dbar}{\overline{\partial}}
\newcommand{\e}{\varepsilon}
\newcommand{\ome}{\widetilde{\omega}}
\newcommand{\I}[1]{\mathcal{J}(#1)}

\newcommand{\lla}[0]{{\langle\!\hspace{0.02cm} \!\langle}}
\newcommand{\rra}[0]{{\rangle\!\hspace{0.02cm}\!\rangle}}

\newtheorem{thm}{Theorem}[section]
\newtheorem{lem}[thm]{Lemma}
\newtheorem{cor}[thm]{Corollary}

\newtheorem{prop}[thm]{Proposition}
\newtheorem{prob}[thm]{Problem}
\newtheorem*{claim}{Claim}

\newtheorem{theorema}{Theorem}

\theoremstyle{definition}
\newtheorem{defn}[thm]{Definition}
\newtheorem{rem}[thm]{Remark}
\newtheorem*{ack}{Acknowledgments}

\newtheorem{step}{Step}
 
\newtheorem{ex}[thm]{Example}

\makeatletter
    
    \@addtoreset{equation}{section}
\makeatother

\begin{document}
\bibliographystyle{amsalpha+}

\maketitle

\begin{abstract}
We formulate and establish a generalization of Koll\'ar's injectivity 
theorem for adjoint bundles twisted by suitable multiplier ideal sheaves. 
As applications, we generalize 
Koll\'ar's torsion-freeness, 
Koll\'ar's vanishing theorem, 
and a generic vanishing theorem  for pseudo-effective line bundles. 
Our approach is not Hodge theoretic but analytic, 
which enables us to treat singular Hermitian metrics with 
nonalgebraic singularities. 
For the proof of the main injectivity theorem, 
we use $L^{2}$-harmonic forms on noncompact K\"ahler manifolds. 
For applications, 
we prove a Bertini-type theorem on the restriction of multiplier 
ideal sheaves to general members of free linear systems. 
\end{abstract}

\tableofcontents
\section{Introduction}\label{f-sec1}

The Kodaira vanishing theorem \cite{kodaira} 
is one of the most celebrated results in complex geometry, 
and it has been generalized to several significant results; 
for example, the Kawamata--Viehweg vanishing theorem, 
the Nadel vanishing theorem, 
Koll\'ar's injectivity theorem 
(see \cite[Chapter 3]{fujino-foundation}). 
Kodaira's original proof is based on the theory of harmonic (differential) forms, 
and has currently been developed to two approaches from different 
perspectives:~One is the Hodge theoretic approach, which is 
algebro-geometric theory based on Hodge structures and spectral 
sequences. The other is the transcendental approach, 
which is an analytic theory focusing on 
harmonic forms and $L^2$-methods for $\dbar$-equations. 
These approaches have been nourishing each other in the last decades. 

As is well known, the Kawamata--Viehweg vanishing theorem 
plays a crucial role in the theory of minimal models for higher-dimensional 
complex algebraic varieties with only mild singularities. 
Now some generalizations of Koll\'ar's injectivity theorem allow us to extend 
the framework of the minimal model program to  highly singular varieties 
(see  \cite{Amb03}, \cite{Amb14}, 
\cite{esnault-viehweg}, \cite{fujino-pja}, \cite{fujino-quasi}, 
\cite{fujino-funda}, 
\cite{fujino-vanishing}, \cite{fujino-injectivity}, 
\cite{fujino-slc}, \cite{fujino-foundation}, \cite{fujino-kodaira}, 
\cite{fujino-vani-semi}, \cite{fujino-slc-surface}, 
\cite{fujino-kollar-type}). 
The reader can find various vanishing theorems and their applications 
in the minimal model program in \cite[Chapters 3 and 6]{fujino-foundation}. 
Koll\'ar's original injectivity theorem, which is one of 
the most important generalizations of the Kodaira vanishing 
theorem, was first established by using the Hodge theory 
(see \cite{kollar-higher1}). 
The following theorem, which is a special case of \cite[Theorem 3.16.2]{fujino-foundation}, 
is obtained from the theory of mixed Hodge structures on 
cohomology with compact support. 

\begin{thm}[Injectivity theorem for log canonical pairs]\label{f-thm1.1}
Let $D$ be a simple normal crossing divisor 
on a smooth projective variety $X$ 
and $F$ be a semiample line bundle on $X$. 
Let $s$ be a nonzero global section of a positive multiple $F^{\otimes m}$ 
such that the zero locus $s^{-1}(0)$ 
contains no log canonical centers of the log canonical 
pair $(X,D)$. 
Then the  map  
\begin{equation*}
\times s: 
H^{i}(X, K_{X} \otimes D \otimes F) 
\to 
H^{i}(X, K_{X} \otimes D \otimes F^{\otimes m+1})
\end{equation*}
induced by $\otimes s$ is injective for every $i$. 
Here $K_{X}$ denotes the canonical bundle of $X$. 
\end{thm}

The Hodge theoretic approach for Theorem \ref{f-thm1.1} is algebro-geometric. 
For the proof, 
we first take a suitable resolution of singularities 
and then take a cyclic cover. 
After that, we apply the $E_1$-degeneration of 
a Hodge to de Rham type spectral sequence coming from 
the theory of mixed Hodge structures on cohomology with 
compact support. 
In this proof, we do not directly use analytic arguments; 
on the contrary, we have no analytic proof for Theorem \ref{f-thm1.1}. 
This indicates that a precise relation between 
the Hodge theoretic approach and the transcendental method 
is not clear yet and is still mysterious.  
There is room for further research from the analytic viewpoint. 
In this paper, we pursue the transcendental approach 
for vanishing theorems instead of the Hodge theoretic approach.

A transcendental approach for Koll\'ar's important work (see \cite{kollar-higher1}) 
was first given by Enoki, 
which improves  Koll\'ar's original injectivity theorem 
to semipositive line bundles on compact K\"ahler manifolds 
as an easy application of the theory of harmonic forms. 
After Enoki's work,  
several authors obtained some generalizations of Koll\'ar's injectivity theorem 
from the analytic viewpoint,  based on the theory of $L^2$-harmonic forms 
(see, for example, \cite{enoki}, 
\cite{takegoshi}, \cite{ohsawa}, \cite{fujino-osaka}, 
\cite{fujino-crelle}, 
\cite{matsumura1}, \cite{matsumura2}, and \cite{matsumura4}). 
Based on the same philosophy, 
it is natural to expect  
Theorem \ref{f-thm1.1}  to hold in the complex analytic setting. 
However, as we mentioned above, 
there is no analytic proof for Theorem \ref{f-thm1.1}. 
Difficulties lie in that 
the usual $L^{2}$-method does not work for log canonical singularities, and 
that no transcendental methods are corresponding 
to the theory of mixed Hodge structures 
(see \cite{matsumura8, noguchi, liu-rao-wan} for some approaches). 
The transcendental method often provides some  powerful tools 
not only in complex geometry but also in algebraic geometry. 
Therefore it is  of interest to 
study various vanishing theorems and related topics 
by using the transcendental method. 

In this paper, by developing the transcendental approach for vanishing theorems, 
we  prove 
Koll\'ar's injectivity, 
vanishing,  torsion-free theorems,  
and a generic vanishing theorem for 
$K_X\otimes F\otimes \mathcal J(h)$, where 
$K_X$ is the canonical bundle of $X$, 
$F$ is a pseudo-effective line bundle on $X$, and 
$\mathcal J(h)$ is the multiplier ideal sheaf associated with 
a singular Hermitian metric $h$. 
More specifically, this paper contains three main contributions:  
The first contribution is to prove a generalization of Koll\'ar's injectivity theorem 
for adjoint bundles twisted by suitable multiplier ideal sheaves (Theorem \ref{f-thmA}). 
The second contribution is to establish a Bertini-type theorem 
on the restriction of multiplier ideal sheaves (Theorem \ref{f-thm1.10}). 
Theorem \ref{f-thm1.10} provides a useful tool and 
enables us to use the inductive argument on dimension. 
The third contribution is to deduce various results related to vanishing theorems 
as applications of Theorem \ref{f-thm1.10} and Theorem \ref{f-thmA}, 
(Theorems \ref{f-thmB}, \ref{f-thmC}, \ref{f-thmD}, \ref{f-thmE}, and \ref{f-thmF}). 
Since we adopt the transcendental method, 
we can formulate all the results  for  
singular Hermitian metrics and (quasi-)plurisubharmonic functions 
with {\textit{arbitrary}} singularities. 
This is one of the main advantages of our approach in this paper. 
The Hodge theoretic approach explained before does not 
work for singular Hermitian metrics with nonalgebraic singularities. 
Furthermore, we sometimes have to deal with 
singular Hermitian metrics with nonalgebraic singularities 
for several important applications in birational geometry  
even when we consider problems in algebraic geometry 
(see, for example, \cite{siu}, \cite{paun}, \cite{dhp}, \cite{gongyo-matsumura}, 
and \cite{lazic-peternell}). 
Therefore, it is worth formulating and proving various 
results for singular Hermitian metrics with arbitrary singularities 
although they are much more complicated than 
singular Hermitian metrics with only algebraic singularities. 
 
\subsection{Main results}\label{f-subsec1.1}
Here, we explain the main results of this paper 
(Theorems \ref{f-thmA}, \ref{f-thmB}, \ref{f-thmC}, 
\ref{f-thmD}, \ref{f-thmE}, \ref{f-thmF}, and Theorem \ref{f-thm1.10}). 
Theorem \ref{f-thmA} and Theorem \ref{f-thm1.10} play important roles in this paper, 
and other results follow from Theorem \ref{f-thmA} and Theorem \ref{f-thm1.10} 
(see Proposition \ref{f-prop1.9}). 
We first recall the definition of 
pseudo-effective line bundles on compact complex manifolds. 

\begin{defn}[Pseudo-effective 
line bundles]\label{f-def1.2}
Let $F$ be a holomorphic line bundle on a compact complex 
manifold $X$. 
We say that $F$ is pseudo-effective 
if there exists a singular Hermitian metric 
$h$ on $F$ with $\sqrt{-1}\Theta_h(F)\geq 0$. 
When $X$ is projective, it is well known that 
$F$ is pseudo-effective 
if and only if $F$ is pseudo-effective in the usual sense, that is, 
$F^{\otimes m}\otimes H$ is big for any ample line bundle $H$ on $X$ 
and any positive integer $m$. 
\end{defn}

The first result is an Enoki-type injectivity theorem. 

\begin{theorema}[Enoki-type injectivity]\label{f-thmA}
Let $F$ be a holomorphic line bundle on a compact K\"ahler 
manifold $X$ and let 
$h$ be a singular Hermitian metric on $F$. 
Let $M$ be a holomorphic line bundle on $X$ and 
let $h_M$ be a smooth Hermitian metric on $M$. 
Assume that 
\begin{equation*}
\sqrt{-1}\Theta_{h_M}(M)\geq 0 \quad \text{and}\quad 
\sqrt{-1}(\Theta_h(F)-t \Theta 
_{h_M}(M))\geq 0
\end{equation*}
for some $t>0$. 
Let 
$s$ be a nonzero global section of $M$. 
Then the map 
\begin{equation*}
\times s: 
H^i(X, K_X\otimes F\otimes 
\mathcal J(h))\to 
H^i(X, K_X\otimes F\otimes \mathcal J(h)\otimes 
M)
\end{equation*} 
induced by $\otimes s$ is injective for every $i$, where 
$K_X$ is the 
canonical bundle of $X$ and 
$\mathcal J(h)$ is the multiplier ideal 
sheaf of $h$. 
\end{theorema}

\begin{rem}\label{f-rem1.3}
Let $L$ be a semipositive line bundle on $X$, 
that is, it admits a smooth Hermitian metric with semipositive curvature. 
Let $F= L^{\otimes m}$ 
and $M= L^{\otimes k}$ for positive integers $m$ and $k$. 
Then we obtain Enoki's original injectivity 
theorem (see \cite[Theorem 0.2]{enoki}) from Theorem \ref{f-thmA}. 
\end{rem}

In the case of $M=F$, 
Theorem \ref{f-thmA} has been proved in \cite{matsumura4} 
under the assumption $\sup_{X}|s|_{h} < \infty$. 
This assumption is a natural condition to 
guarantee that the multiplication map $\times s$ is well-defined. 
However, for our applications in this paper, 
we need to formulate Theorem \ref{f-thmA} 
for a different $(M, h_{M})$ from $(F, h)$. 
This formulation, which may look slightly artificial, 
is quite powerful and can produce applications, 
but raises a new difficulty in the proof: 
the set of points $x \in X$ with $\nu(h,x)>0$ is 
not necessarily contained in a proper Zariski closed set, 
although such a situation was excluded in \cite{matsumura4} 
thanks to the assumption $\sup_{X}|s|_{h} < \infty$, 
where $\nu(h,x)$ denotes the Lelong number of the local weight of $h$ at $x$. 
Compared to \cite{matsumura4}, 
Theorem \ref{f-thmA} is novel in the technique to overcome this difficulty 
(see Section \ref{f-sec5} for the technical details), 
and further, it will be generalized 
to certain noncompact manifolds 
along with other techniques (see \cite{matsumura5}).
Note that Theorem \ref{f-thmA} can be seen as a generalization  
not only of Enoki's injectivity theorem but also of the Nadel vanishing theorem. 
In Section \ref{x-sec4}, we will explain how to 
reduce Demailly's original 
formulation of the Nadel vanishing theorem (see Theorem \ref{f-thm1.4} below) to 
Theorem \ref{f-thmA} for the reader's convenience. 

\begin{thm}[{Nadel vanishing theorem due to 
Demailly:~\cite[Theorem 4.5]{demailly-numerical}}]\label{f-thm1.4}
Let $V$ be a smooth projective variety equipped with 
a K\"ahler form $\omega$. 
Let $L$ be a holomorphic line bundle on $V$ and let $h_L$ be 
a singular Hermitian metric on $L$ such that 
$\sqrt{-1}\Theta_{h_L}(L)\geq \varepsilon \omega$ for 
some $\varepsilon >0$. 
Then 
$$
H^i(V, K_V\otimes L\otimes \mathcal J(h_L))=0
$$ 
for every $i>0$,  
where $K_V$ is the 
canonical bundle of $V$ and 
$\mathcal J(h_L)$ is the multiplier ideal 
sheaf of $h_L$. 
\end{thm}

A semiample line bundle is always semipositive. 
Thus, as a direct consequence of Theorem \ref{f-thmA}, 
we obtain Theorem \ref{f-thmB}, which is a generalization of 
Koll\'ar's original injectivity theorem (see \cite{kollar-higher1}). 

\begin{theorema}[Koll\'ar-type injectivity]\label{f-thmB}
Let $F$ be a holomorphic line bundle on a compact K\"ahler manifold $X$ and let 
$h$ be a singular Hermitian metric on $F$ such that 
$\sqrt{-1}\Theta_h (F)\geq 0$. 
Let $N_1$ and $N_2$ be semiample line bundles on 
$X$ and let 
$s$ be a nonzero global section of $N_2$. 
Assume that $N_1^{\otimes a}\simeq N_2^{\otimes b}$ 
for some positive integers $a$ and $b$.  
Then the map 
\begin{equation*}
\times s: 
H^i(X, K_X\otimes F\otimes 
\mathcal J(h)\otimes N_1)\to 
H^i(X, K_X\otimes F\otimes \mathcal J(h)\otimes 
N_1\otimes N_2)
\end{equation*} 
induced by $\otimes s$ is injective for every $i$, 
where $K_X$ is the 
canonical bundle of $X$ and 
$\mathcal J(h)$ is the multiplier ideal 
sheaf of $h$. 
\end{theorema}

\begin{rem}\label{f-rem1.5}
(1) 
Let $X$ be a smooth projective variety and 
$(F, h)$ be a trivial Hermitian line bundle. 
Then we obtain  Koll\'ar's original injectivity theorem (see \cite[Theorem 2.2]
{kollar-higher1}) from Theorem \ref{f-thmB}. 
\\
(2)
For the proof of Theorem \ref{f-thmB}, 
we may assume that $b=1$, that is, 
$N_2\simeq N_1^{\otimes a}$ 
by replacing $s$ with $s^b$. 
We note that the composition 
\begin{equation*}
\begin{split}
H^i(X, K_X\otimes F\otimes 
\mathcal J(h)\otimes N_1)
&\overset{\times s}{\longrightarrow} 
H^i(X, K_X\otimes F\otimes \mathcal J(h)\otimes 
N_1\otimes N_2)\\
&\overset{\times s^{b-1}}{\longrightarrow}
H^i(X, K_X\otimes F\otimes \mathcal J(h)\otimes 
N_1\otimes N_2^{\otimes b})
\end{split}
\end{equation*} 
is the map $\times s^b$ induced by $\otimes s^b$. 
\end{rem}

Theorem \ref{f-thmC} is a generalization of 
Koll\'ar's torsion-free theorem and Theorem \ref{f-thmD} is 
a generalization of Koll\'ar's vanishing theorem (see 
\cite[Theorem 2.1]{kollar-higher1}). 

\begin{theorema}[Koll\'ar-type torsion-freeness]\label{f-thmC} 
Let $f\colon X\to Y$ be a surjective morphism from a 
compact K\"ahler manifold $X$ onto a projective 
variety $Y$. 
Let $F$ be a holomorphic line bundle on $X$ and let 
$h$ be a singular Hermitian metric on $F$ such that 
$\sqrt{-1}\Theta_h (F)\geq 0$. 
Then $$R^if_*(K_X\otimes F\otimes \mathcal J(h))$$ is 
torsion-free for every $i$, where 
$K_X$ is the 
canonical bundle of $X$ and 
$\mathcal J(h)$ is the multiplier ideal 
sheaf of $h$. 
\end{theorema}

\begin{theorema}[Koll\'ar-type vanishing theorem]\label{f-thmD} 
Let $f\colon X\to Y$ be a surjective morphism from a compact 
K\"ahler manifold $X$ onto a projective variety $Y$. 
Let $F$ be a holomorphic line bundle on $X$ and let 
$h$ be a singular Hermitian metric on $F$ such that 
$\sqrt{-1}\Theta_h (F)\geq 0$. 
Let $N$ be a holomorphic line bundle on $X$. 
We assume that there exist 
positive integers $a$ and $b$ and an ample line bundle $H$ on $Y$ 
such that $N^{\otimes a}\simeq 
f^*H^{\otimes b}$. 
Then we obtain that 
\begin{equation*}
H^i(Y, R^jf_*(K_X\otimes F\otimes \mathcal J(h)\otimes N))=0
\end{equation*}
for every $i>0$ and $j$, where 
$K_X$ is the 
canonical bundle of $X$ and 
$\mathcal J(h)$ is the multiplier ideal 
sheaf of $h$. 
\end{theorema}

\begin{rem}\label{f-rem1.6}
(1) If $X$ is a smooth projective variety and $(F, h)$ is trivial, 
then Theorem \ref{f-thmC} is nothing but Koll\'ar's 
torsion-free theorem. 
Furthermore, if $N\simeq f^*H$, that is, 
$a=b=1$, then Theorem \ref{f-thmD} is the Koll\'ar 
vanishing theorem. 
For the details, see \cite[Theorem 2.1]{kollar-higher1}. 
\\
(2) There exists a clever proof of Koll\'ar's torsion-freeness 
by the theory of variations of Hodge structure 
(see \cite{arapura}). 
\\
(3) In \cite{matsumura6}, the second author obtained 
a natural analytic generalization of Koll\'ar's vanishing theorem, 
which corresponds to the case where $h$ is a smooth Hermitian metric 
and contains Ohsawa's vanishing 
theorem (see \cite{ohsawa-vanishing}) as a special case. 
\\
(4) In \cite{fujino-amc}, the first author proved a vanishing theorem 
containing both Theorem \ref{f-thm1.4} and 
Theorem \ref{f-thmD} as special 
cases, which is called the vanishing theorem of Koll\'ar--Nadel type. 
\end{rem}

By combining Theorem \ref{f-thmD} with the Castelnuovo--Mumford regularity, 
we can easily obtain Corollary \ref{f-cor1.7}, which is a complete 
generalization of \cite[Lemma 3.35 and Remark 3.36]{horing}. 
The proof of \cite[Lemma 3.35]{horing} depends on a 
generalization of the Ohsawa--Takegoshi $L^2$ extension theorem. 
We note that H\"oring claims the weak positivity of $f_*(K_{X/Y}\otimes 
F)$ under some extra assumptions by using \cite[Lemma 3.35]{horing}. 
For the details, see \cite[3.H Multiplier ideals]{horing}. 

\begin{cor}\label{f-cor1.7}
Let $f\colon X\to Y$ be a surjective morphism 
from a compact K\"ahler manifold $X$ onto a 
projective variety $Y$. 
Let $F$ be a holomorphic line bundle 
on $X$ and let $h$ be a singular Hermitian metric on $F$ such that 
$\sqrt{-1}\Theta_h(F)\geq 0$. 
Let $H$ be an ample and globally generated line bundle on $Y$. 
Then 
$$
R^if_*(K_X\otimes F\otimes \mathcal J(h))\otimes 
H^{\otimes m}
$$ 
is globally generated for every $i\geq 0$ and 
$m\geq \dim Y+1$, where 
$K_X$ is the canonical bundle of $X$ and 
$\mathcal J(h)$ is the multiplier ideal sheaf of $h$. 
\end{cor}

As a direct consequence of Theorem \ref{f-thmD}, we obtain Theorem \ref{f-thmE}. 
See Definition \ref{f-def1.8} for the definition of 
GV-sheaves in the sense of Pareschi and Popa 
and see \cite[Theorem 25.5 and Definition 26.3]{schnell} for the details of GV-sheaves. 

\begin{theorema}[GV-sheaves]\label{f-thmE}
Let $f\colon X\to A$ be a morphism from a 
compact K\"ahler manifold 
$X$ to an Abelian variety $A$. 
Let $F$ be a holomorphic line bundle on $X$ and let 
$h$ be a singular Hermitian metric on $F$ such that 
$\sqrt{-1}\Theta_h (F)\geq 0$. 
Then 
$$R^if_*(K_X\otimes F\otimes \mathcal J(h))
$$ 
is a GV-sheaf for every $i$, 
where $K_X$ is the canonical bundle of $X$ and 
$\mathcal J(h)$ is the multiplier ideal sheaf of $h$.  
\end{theorema} 

\begin{defn}[{GV-sheaves in the sense of Pareschi and 
Popa:~\cite{pp}}]\label{f-def1.8} 
Let $A$ be an Abelian variety. 
A coherent sheaf $\mathcal F$ on $A$ is said to be a GV-sheaf 
if 
$$
\mathrm{codim}_{\Pic^0(A)}\{ L\in \Pic^0(A) \, | 
\, H^i(A, \mathcal F\otimes L)\ne 0\} \geq i 
$$ 
for every $i$. 
\end{defn}

The final one is a generalization of the generic vanishing theorem 
(see \cite{green-lazarsfeld}, \cite{hacon}, \cite{pp}). 
The formulation of Theorem \ref{f-thmF} is closer to 
\cite{hacon} and \cite{pp} than to the original 
generic vanishing theorem by Green and Lazarsfeld 
in \cite{green-lazarsfeld}. 

\begin{theorema}[Generic vanishing theorem]\label{f-thmF}
Let $f\colon X\to A$ be a morphism from a 
compact K\"ahler manifold 
$X$ to an Abelian variety $A$. 
Let $F$ be a holomorphic line bundle on $X$ and let 
$h$ be a singular Hermitian metric on $F$ such that 
$\sqrt{-1}\Theta_h (F)\geq 0$. 
Then 
\begin{equation*}
\begin{split}
\mathrm{codim}_{\Pic^0(A)}\{L\in \Pic^0(A) \, 
|\, H^i(X, K_X\otimes F\otimes \mathcal J(h)\otimes 
f^*L)\ne 0\} \geq i -(\dim X -\dim f(X)) 
\end{split}
\end{equation*}
for every $i\geq 0$, where 
$K_X$ is the 
canonical bundle of $X$ and 
$\mathcal J(h)$ is the multiplier ideal 
sheaf of $h$. 
\end{theorema}

The main results explained above 
are closely related to each other. 
The following proposition, which is 
also one of the main contributions in this paper, 
shows several relations among them. 
From Proposition \ref{f-prop1.9}, 
we see that it is sufficient to prove Theorem \ref{f-thmA}.
The proof of Proposition \ref{f-prop1.9} will be given 
in Section \ref{x-sec4}. 

\begin{prop}\label{f-prop1.9}
We have the following relations among the 
above theorems. 
\begin{itemize}
\item[(i)] Theorem \ref{f-thmA} implies Theorem \ref{f-thmB}. 
\item[(ii)] Theorem \ref{f-thmB} is equivalent to 
Theorem \ref{f-thmC} and Theorem \ref{f-thmD}. 
\item[(iii)] Theorem \ref{f-thmD} implies Theorem \ref{f-thmE}. 
\item[(iv)] Theorem \ref{f-thmC} and Theorem 
\ref{f-thmE} imply Theorem \ref{f-thmF}. 
\end{itemize}
\end{prop}
 
A key ingredient of Proposition \ref{f-prop1.9} is 
the following theorem, 
which can be seen as a Bertini-type 
theorem on the restriction of multiplier 
ideal sheaves to general members of free linear systems. 
Theorem \ref{f-thm1.10} enables us to use the 
inductive argument on dimension. 
We remark that $\mathcal{G}$ in Theorem \ref{f-thm1.10} 
is not always an intersection of countably many Zariski open sets 
(see Example \ref{x-ex3.10}). 
The proof of Theorem \ref{f-thm1.10}, 
which is quite technical, will be given in Section \ref{x-sec3}

\begin{thm}[Density of good divisors:~Theorem \ref{x-thm3.6}]
\label{f-thm1.10} 
Let $X$ be a compact complex manifold, 
let $\Lambda$ be a free linear system on $X$ with $\dim \Lambda \geq 1$, 
and let $\varphi$ be a quasi-plurisubharmonic function on $X$. 
We put 
\begin{equation*}
\mathcal G :=\{ H\in \Lambda \, |\,  
{\text{$H$ is smooth and $\mathcal J(\varphi|_H)=
\mathcal J(\varphi)|_H$}}\}. 
\end{equation*}
Then $\mathcal G$ is dense in $\Lambda$ 
in the classical topology, that is, the Euclidean topology. 
\end{thm}

Although the above formulation is sufficient for our applications, 
it is of independent interest to find a more precise formulation. 
The following problem, posed by S\'ebastien Boucksom, 
is reasonable from the viewpoint of 
Berndtsson's complex Prekopa theorem (see \cite{Be}). 

\begin{prob}
\label{f-prob1.11}
In Theorem \ref{f-thm1.10}, is the complement $\Lambda\setminus 
\mathcal G$ a pluripolar subset of $\Lambda$? 
\end{prob}

All the results explained above 
hold even if we replace $K_X$ with $K_X\otimes E$, 
where $E$ is any Nakano semipositive vector bundle on $X$. 
We will explain Theorem \ref{f-thm1.12} in Section \ref{f-sec6}. 

\begin{thm}[Twists by Nakano semipositive vector bundles]\label{f-thm1.12}
Let $E$ be a Nakano semipositive vector bundle on a compact 
K\"ahler manifold $X$. 
Then Theorems \ref{f-thmA}, \ref{f-thmB}, 
\ref{f-thmC}, \ref{f-thmD}, \ref{f-thmE}, \ref{f-thmF}, Theorem 
\ref{f-thm1.4}, Corollary \ref{f-cor1.7}, and 
Proposition \ref{f-prop1.9} hold even when $K_X$ is replaced 
with $K_X\otimes E$. 
\end{thm}

In this paper, we assume that all the varieties and manifolds are 
compact and connected for simplicity. 
We summarize the contents of this paper. 
In Section \ref{f-sec2}, we recall some basic definitions 
and collect several preliminary lemmas. 
Section \ref{x-sec3} is devoted to the 
proof of Theorem \ref{f-thm1.10}. 
Theorem \ref{f-thm1.10} plays a crucial role in the proof of 
Proposition \ref{f-prop1.9}. 
In Section \ref{x-sec4}, we prove Proposition \ref{f-prop1.9} and  
Corollary \ref{f-cor1.7}, and explain how to reduce 
Theorem \ref{f-thm1.4} to Theorem \ref{f-thmA}. 
By these results, we see that 
all we have to do is to establish Theorem 
\ref{f-thmA}. 
In Section \ref{f-sec5}, we give a detailed proof of Theorem \ref{f-thmA}. 
In the final section:~Section \ref{f-sec6}, 
we explain how to modify the arguments used before 
for the proof of Theorem \ref{f-thm1.12}. 

After the authors put a preprint version of this paper on arXiv, 
some further generalizations of Theorem \ref{f-thmA} have been studied 
in \cite{matsumura5}, \cite{CDM}, \cite{ZZ}, 
and a relative version of Theorem \ref{f-thm1.10} has been established in 
\cite{fujino-relative}, by developing the techniques in this paper.  
See \cite{takegoshi}, \cite{fujino-crelle}, \cite{matsumura5}, \cite{CDM}, 
\cite{fujino-relative} for some injectivity, torsion-free, and vanishing theorems 
for noncompact manifolds. 

\begin{ack}
The authors would like to thank Professor Toshiyuki Sugawa 
for giving them the reference on Example \ref{x-ex3.10} 
and Professor Taro Fujisawa for his warm encouragement. 
Furthermore, they are deeply grateful to Professor S\'ebastien Boucksom 
for kindly suggesting Problem \ref{f-prob1.11}. 
The first author thanks Takahiro Shibata for discussion. 
He was partially supported by 
Grant-in-Aid for Young Scientists (A) 24684002 from JSPS and 
by JSPS KAKENHI Grant 
Numbers JP16H03925, JP16H06337, 
JP19H01787, JP20H00111, JP21H00974. 
The second author was partially supported by 
Grand-in-Aid for Young Scientists (A) $\sharp$17H04821, 
Grand-in-Aid for Scientific Research (B) $\sharp$ 21H00976, 
and Fostering Joint International Research (A) $\sharp$ 19KK0342 from JSPS. 
\end{ack}

\section{Preliminaries}\label{f-sec2}

We briefly review the definition of singular Hermitian metrics, 
(quasi-)plurisubharmonic functions, and Nadel's multiplier ideal sheaves. 
See \cite{demailly} for the details. 

\begin{defn}[Singular Hermitian metrics and curvatures]\label{f-def2.1} 
Let $F$ be a holomorphic line bundle
on a complex manifold $X$. 
A {\textit{singular Hermitian metric}}
on $F$ is a metric $h$ which is given in 
every trivialization $\theta\colon F|_{\Omega}\simeq \Omega\times 
\mathbb C$ 
by 
\begin{equation*}
| \xi |_{h} =|\theta(\xi)|e^{-\varphi} \text{ on } \Omega, 
\end{equation*}
where $\xi$ is a section of $F$ on $\Omega$ and 
$\varphi \in L^1_{\mathrm{loc}}(\Omega)$ 
is an arbitrary function. 
Here $L^1_{\mathrm{loc}}(\Omega)$ is the space
of locally integrable functions on $\Omega$. 
We usually call $\varphi$ 
the weight function of the
metric with respect to the trivialization $\theta$. 
The curvature of a singular Hermitian metric $h$ is 
defined by 
\begin{equation*}
\sqrt{-1}\Theta_{h}(F):=2\deldel \varphi, 
\end{equation*}
where $\varphi$ is a weight function and 
$\deldel \varphi$ is taken in the sense of currents. 
It is easy to see that the right-hand side does not depend on 
the choice of trivializations.
\end{defn}

The notion of multiplier ideal sheaves introduced by Nadel 
plays an important role in the recent developments of 
complex geometry and algebraic geometry. 

\begin{defn}[(Quasi-)plurisubharmonic functions and 
multiplier ideal sheaves]\label{f-def2.2} 
A function $u\colon \Omega\to [-\infty, \infty)$ defined on an open set 
$\Omega\subset \mathbb C^n$ is said to be {\textit{plurisubharmonic}} if 
\begin{itemize}
\item[$\bullet$] $u$ is upper semicontinuous, and 
\item[$\bullet$] for every complex line $L\subset \mathbb C^n$, 
the restriction $u|_{\Omega\cap L}$ to $L$ 
is subharmonic on $\Omega\cap L$, that is, 
for every $a\in \Omega$ and $\xi \in \mathbb C^n$ 
satisfying $|\xi|<d(a, \Omega^c)$, the function $u$ satisfies 
the mean inequality 
\begin{equation*}
u(a)\leq \frac{1}{2\pi} \int ^{2\pi}_{0} u(a+e^{i\theta}\xi)\,d\theta. 
\end{equation*}
\end{itemize} 

Let $X$ be a complex manifold. 
A function $\varphi \colon X\to [-\infty, \infty)$ is said to be 
{\textit{plurisubharmonic on $X$}} if 
there exists an open cover $\{U_{i}\}_{i \in I}$ of $X$ such that 
$\varphi|_{U_i}$ is plurisubharmonic on $U_i$ 
($\subset \mathbb C^n$) for every $i$. 
We can easily see that 
this definition is independent of the choice 
of open covers. 
A {\textit{quasi-plurisubharmonic}} function is a function $\varphi$ which 
is locally equal to the sum of a plurisubharmonic 
function and of a smooth function. 
If $\varphi$ is a quasi-plurisubharmonic function on a complex manifold $X$, 
then the multiplier ideal sheaf 
$\mathcal J(\varphi)\subset \mathcal O_X$ is 
defined by 
\begin{equation*}
\Gamma (U, \mathcal J(\varphi))
:=\{f\in \mathcal O_X(U)\, |\, |f|^2e^{-2\varphi}\in 
L^1_{\mathrm{loc}}(U) \}
\end{equation*}
for every open set $U\subset X$. 
Then it is known that 
$\mathcal J(\varphi)$ is a coherent ideal sheaf 
(see, for example, \cite[(5.7) Lemma]{demailly}). 
Let $S$ be a complex submanifold 
of $X$. Then the restriction $\mathcal J(\varphi)|_S$ 
of the multiplier ideal sheaf $\mathcal J(\varphi)$ to $S$ is 
defined by the image of $\mathcal J(\varphi)$ under 
the natural surjective 
morphism $\mathcal O_X\to \mathcal O_S$, 
that is, 
\begin{equation*}
\mathcal J(\varphi)|_S=\mathcal J(\varphi) / \mathcal J(\varphi)
\cap\mathcal I_S, 
\end{equation*} 
where $\mathcal I_S$ is the defining ideal sheaf of 
$S$ on $X$. 
We note that the restriction $\mathcal J(\varphi)|_S$ 
does not always coincide with $\mathcal J(\varphi)\otimes 
\mathcal O_S=\mathcal J(\varphi)/\mathcal J(\varphi) \cdot \mathcal I_S$. 
\end{defn}

We have already used $\mathcal J(h)$ 
in theorems in Section \ref{f-sec1}. 

\begin{defn}\label{f-def2.3}
Let $F$ be a holomorphic line bundle on a complex manifold $X$ and 
let $h$ be a singular Hermitian metric on $F$.  
We assume $\sqrt{-1}\Theta_h(F)\geq \gamma$ 
for some smooth $(1, 1)$-form $\gamma$ on $X$.
We fix a smooth Hermitian metric $h_{\infty}$ on $F$. 
Then we can write $h=h_{\infty} e^{-2\psi}$ for some 
$\psi \in L^1_{\mathrm{loc}}(X)$.  
Then 
$\psi$ coincides with a quasi-plurisubharmonic 
function $\varphi$ on $X$ almost everywhere. 
We define the multiplier ideal sheaf 
$\mathcal J(h)$ of $h$ by $\mathcal J(h):=\mathcal J(\varphi)$. 
\end{defn}

We close this section with the following lemmas, 
which will be used in the proof of Theorem \ref{f-thmA} 
in Section \ref{f-sec5}. 

\begin{lem}[{\cite[Proposition 1.1]{ohsawa-norm}}]\label{f-lem2.4}
Let $\omega$ and  $\ome$ be positive $(1,1)$-forms on an 
$n$-dimensional complex manifold with $\ome \geq \omega$.  
If $u$ is an $(n,q)$-form, 
then $|u|^{2}_{\ome}\, dV_{\ome} \leq |u|^{2}_{\omega}\,  dV_{\omega} $. 
Furthermore, if $u$ is an $(n,0)$-form, 
then $|u|^{2}_{\ome}\,  dV_{\ome} = |u|^{2}_{\omega}\,  dV_{\omega}$.
Here $|u|_{\omega}$ $($resp.~$|u|_{\ome}$$)$ is the pointwise norm of $u$ 
with respect to $\omega$ $($resp.~$\ome$$)$ and 
$dV_{\omega}$ $($resp.~$dV_{\ome}$$)$ is the volume form defined by 
$dV_{\omega}:=\omega^{n}/n!$ $($resp.~$dV_{\ome}:=\ome^{n}/n!$$)$. 
\end{lem}
\begin{proof}
This lemma follows from simple computations. 
Thus, we omit the proof. 
\end{proof}

\begin{lem}\label{f-lem2.5}
Let $\varphi\colon \mathcal{H}_{1} \to \mathcal{H}_{2}$ be 
a bounded operator $($continuous linear map$)$ between Hilbert spaces
$\mathcal{H}_{1}, \mathcal{H}_{2}$. 
If $\{w_{k}\}_{k=1}^{\infty}$ weakly converges to $w$ in $\mathcal{H}_{1}$, 
then $\{\varphi(w_{k})\}_{k=1}^{\infty}$ weakly converges to $\varphi(w)$. 
\end{lem}
\begin{proof}
By taking the adjoint operator $\varphi^*$, 
for every $v \in \mathcal{H}_{2}$, 
we have 
$$
\lla \varphi(w_{k}), v \rra_{\mathcal{H}_{2}}
= \lla w_{k}, \varphi^*(v) \rra_{\mathcal{H}_{1}}
\to \lla w, \varphi^*(v) \rra_{\mathcal{H}_{1}}
=\lla \varphi(w), v \rra_{\mathcal{H}_{2}}. 
$$
This completes the proof. 
\end{proof}

\begin{lem}\label{f-lem2.6}
Let $L$ be a closed subspace in a Hilbert space $\mathcal{H}$. 
Then $L$ is closed with respect to the weak topology of $\mathcal{H}$, 
that is, if a sequence $\{w_{k}\}_{k=1}^{\infty}$ in $L$ 
weakly converges to $w$, then the weak limit $w$ belongs to $L$. 
\end{lem}
\begin{proof}
By the orthogonal decomposition, there exists a closed subspace $M$ 
such that $L=M^{\perp}$. 
Then it follows that $w \in M^{\perp}=L$ 
since $
0=\lla w_{k}, v \rra_{\mathcal{H}}
\to \lla w, v \rra_{\mathcal{H}}
$
for any $v \in M$. 
\end{proof}

\section{Restriction lemma}\label{x-sec3}
This section is devoted to the proof of Theorem \ref{f-thm1.10} (see 
Theorem \ref{x-thm3.6}), 
which will play a crucial role in the proof of Proposition \ref{f-prop1.9}. 
The following lemma is a direct consequence of 
the Ohsawa--Takegoshi $L^2$ extension theorem (see \cite[Theorem]
{ohsawa-takegoshi}). 

\begin{lem}\label{x-lem3.1} 
Let $X$ be a complex manifold 
and let $\varphi$ be a quasi-plurisubharmonic function on $X$. 
We consider a filtration  
\begin{equation*}
F_k \subset F_{k-1}\subset \cdots \subset F_1 \subset 
F_0:=X, 
\end{equation*}
where $F_i$ is a smooth hypersurface of $F_{i-1}$ for every 
$i$. 
Then we obtain that 
\begin{equation*}
\mathcal J(\varphi|_{F_k})\subset 
\mathcal J(\varphi|_{F_{k-1}})|_{F_k}\subset 
\cdots \subset \mathcal J(\varphi|_{F_1})|_{F_k}\subset 
\mathcal J(\varphi)|_{F_k}. 
\end{equation*} 
\end{lem}
\begin{proof} 
This immediately follows from the Ohsawa--Takegoshi $L^2$ extension theorem. 
\end{proof}

The following lemma is a key ingredient of 
the proof of Theorem \ref{f-thm1.10} (see Theorem \ref{x-thm3.6}). 

\begin{lem}\label{x-lem3.2}
Let $X$ and $\varphi$ be as in Theorem \ref{f-thm1.10}. 
Let $H_i$ be a Cartier divisor on $X$ for $1\leq i\leq k$. 
We assume the following condition:  
\begin{itemize}
\item[$\spadesuit$]\label{x-spade} 
The divisor $\sum _{i=1}^kH_i$ is a simple normal crossing 
divisor on $X$. 
Moreover, 
for every $1\leq i_1<i_2<\cdots < i_l\leq k$ and any $P\in H_{i_1}\cap 
H_{i_2}\cap \cdots \cap H_{i_l}$, 
the set $\{f_{i_1}, f_{i_2}, \cdots, f_{i_l}\}$ is a regular sequence 
for $\mathcal O_{X, P}/\mathcal J(\varphi)_P$, where 
$f_i$ is a $($local$)$ defining equation of $H_i$ for every $i$. 
\end{itemize}
 Furthermore, we assume that 
 $\mathcal J(\varphi|_{F_k})=\mathcal J(\varphi)|_{F_k}$ 
holds, 
where $F_i:=H_1\cap H_2\cap \cdots \cap H_i$ for $1 \leq i \leq k$. 
Then  for every $j$, the equality 
$\mathcal J(\varphi|_{F_j})=\mathcal J(\varphi)|_{F_j}$ holds 
on a neighborhood of $F_k$ in $F_{j}$. 
\end{lem}

Before we prove Lemma \ref{x-lem3.2}, 
we make some remarks to help the reader understand 
condition $\spadesuit$. 

\begin{rem}\label{x-rem3.3}
(1) Let $(A, \mathfrak m)$ be a local ring and let $M$ be a finitely 
generated (nonzero) $A$-module. 
Let $\{ x_1, \ldots, x_r \}$ be a sequence of elements of $\mathfrak m$. 
We put $M_0=M$ and $M_i=M/x_1M+\cdots +x_i M$. 
Then $\{x_1, \ldots, x_r\}$ is said to be a {\textit{regular sequence}} 
for $M$ if $\times x_{i+1}\colon M_i\to M_i$ is injective 
for every $0\leq i\leq r-1$. 

(2) Condition $\spadesuit$ in Lemma \ref{x-lem3.2} does not depend on 
the order of $\{H_1, H_2, \cdots, H_k\}$ 
(see, for example, \cite[Theorem 16.3]{matsumura-book} and 
\cite[Chapter III, Corollary (3.5)]{ak}). 

(3) Let $\mathcal F$ be a coherent analytic sheaf on a 
compact complex manifold $X$. 
Then there exists a finite family $\{Y_i\}_{i\in I}$ of 
irreducible analytic subsets of $X$ such that 
$$\mathrm{Ass}_{\mathcal O_{X, x}}(\mathcal F_x)=\{ \mathfrak p_{x, 1}, 
\ldots, \mathfrak p_{x, r(x)}\}, 
$$ 
where $\mathfrak p_{x, 1}, 
\ldots, \mathfrak p_{x, r(x)}$ are prime ideals of $\mathcal O_{X, x}$ associated 
to the irreducible components of the germs $x\in Y_i$ 
(see, for example \cite[(I.6) Lemma]{manaresi}). 
Note that $Y_i$ is called an 
analytic subset associated with $\mathcal F$. 
In this paper, we simply say that $Y_i$ is an {\textit{associated prime}} 
of $\mathcal F$ if there is no risk of confusion. 
Then condition $\spadesuit$ is equivalent to the following condition: 
\begin{itemize}
\item 
The divisor $\sum _{i=1}^kH_i$ is a simple normal crossing divisor on $X$. 
Moreover, 
for every $1\leq i_1<i_2<\cdots < i_{l-1}< i_l\leq k$,  
the divisor $H_{i_l}$ contains no associated primes of 
$\mathcal O_X/\mathcal J(\varphi)$ and 
$\mathcal O_{H_{i_1}\cap \cdots \cap H_{i_{l-1}}}/ 
\mathcal J(\varphi)|_{H_{i_1}\cap \cdots \cap H_{i_{l-1}}}$. 
\end{itemize} 

(4) \eqref{x-3.1} below may be helpful to understand condition $\spadesuit$. 
We put $H_{i_1 \cdots i_m}:=H_{i_1}\cap 
\cdots \cap H_{i_m}$ for every $1\leq i_1<\cdots < i_m\leq k$. 
Then we can inductively check that 
\begin{equation*}
0\to \mathcal J(\varphi)|_{H_{i_1 \cdots i_{l-1}}}\otimes 
\mathcal O_{H_{i_1 \cdots i_{l-1}}}(-H_{i_l})\to 
\mathcal J(\varphi)|_{H_{i_1 \cdots i_{l-1}}}
\to \mathcal J(\varphi)|_{H_{i_1 \cdots i_l}}\to 0
\end{equation*}
is exact and that 
\begin{equation}\label{x-3.1}
\begin{split}
0&\to \left (\mathcal O_{H_{i_1 \cdots i_{l-1}}}
/\mathcal J(\varphi)|_{H_{i_1 \cdots i_{l-1}}}
\right) \otimes \mathcal O_{H_{i_1 \cdots i_{l-1}}}(-H_{i_l})
\\&\to 
\mathcal O_{H_{i_1 \cdots i_{l-1}}}
/\mathcal J(\varphi)|_{H_{i_1 \cdots i_{l-1}}}
\to 
\mathcal O_{H_{i_1 \cdots i_l}}
/\mathcal J(\varphi)|_{H_{i_1 \cdots i_l}}\to 0
\end{split}
\end{equation}
is also exact (see \eqref{x-3.3} 
and \eqref{x-3.4} in the proof of Lemma \ref{x-lem3.2}). 
\end{rem}

\begin{proof}[Proof of Lemma \ref{x-lem3.2}] 
By condition $\spadesuit$, 
the morphism $\gamma$ in the 
following commutative diagram is injective. 
\begin{equation}\label{x-3.2}
\xymatrix{
 & 0\ar[d] & 0\ar[d]&\\
0 \ar[r]& \mathcal J(\varphi) \otimes \mathcal O_X(-H_1)
\ar[d]
\ar[r]^{\quad \quad \quad \alpha}
&\mathcal J(\varphi)\ar[d]\ar[r]& \Coker \alpha 
\ar[r]\ar[d]^{\beta}& 0 \\ 
0 \ar[r]& \mathcal O_X(-H_1)\ar[d]
\ar[r]
&\mathcal O_X \ar[d]\ar[r]& \mathcal O_{H_1}\ar[r]& 0 \\
&\left(\mathcal O_X/\mathcal J(\varphi)\right)\otimes 
\mathcal O_X(-H_1) \ar[r]^{\quad \quad \quad \gamma}
\ar[d]& \mathcal O_X/\mathcal J(\varphi)\ar[d]&\\
& 0& 0&
}
\end{equation}
Therefore $\beta$ is also injective. 
This implies that $\Coker \alpha=\mathcal J(\varphi)|_{H_1}$ by 
definition. 
Thus, we obtain the following short exact sequence: 
\begin{equation*}
0\to \mathcal J(\varphi)\otimes 
\mathcal O_X(-H_1)\to 
\mathcal J(\varphi)\to \mathcal J(\varphi)|_{H_1}\to 0.
\end{equation*} 
We also obtain the following short exact sequence: 
\begin{equation*}
0\to \left(\mathcal O_X/\mathcal J(\varphi)\right)\otimes 
\mathcal O_X(-H_1)\overset{\gamma}{\to} \mathcal O_X/\mathcal J(\varphi) \to 
\mathcal O_{H_1}/\mathcal J(\varphi)|_{H_1}\to 0
\end{equation*} 
by the above 
big commutative diagram.  
Similarly, by condition $\spadesuit$, 
we can inductively check that 
\begin{equation}\label{x-3.3}
0\to \mathcal J(\varphi)|_{F_i}\otimes 
\mathcal O_{F_i}(-H_{i+1})\to 
\mathcal J(\varphi)|_{F_i}\to \mathcal J(\varphi)|_{F_{i+1}}\to 0
\end{equation} 
and 
\begin{equation}\label{x-3.4}
0\to \left(\mathcal O_{F_i}/\mathcal J(\varphi)|_{F_i}\right)
\otimes \mathcal O_{F_i}(-H_{i+1})\to 
\mathcal O_{F_i}/\mathcal J(\varphi)|_{F_i}\to 
\mathcal O_{F_{i+1}}/\mathcal J(\varphi)|_{F_{i+1}}\to 0
\end{equation}
are exact for every $1\leq i\leq k-1$. 
For $0 \leq i \leq k-1$, we consider the following commutative diagram: 
$$
\xymatrix{
 & 0\ar[d] & 0\ar[d]&\\
0 \ar[r]& \mathcal J(\varphi|_{F_i}) \otimes \mathcal O_{F_i}(-H_{i+1})
\ar[d]_{a_i}
\ar[r]
&\mathcal J(\varphi)|_{F_i} \otimes 
\mathcal O_{F_i}(-H_{i+1})\ar[d]\ar[r]& \Coker b_i\otimes \mathcal O_{F_i}
(-H_{i+1})\ar[r]\ar[d]^{d_i}& 0 \\ 
0 \ar[r]& \mathcal J(\varphi|_{F_i}) 
\ar[d]
\ar[r]^{b_i}
&\mathcal J(\varphi)|_{F_i} \ar[d]\ar[r]& \Coker b_i\ar[r]& 0 \\
&\Coker a_i \ar[r]^{c_i}\ar[d]& \mathcal J(\varphi)|_{F_{i+1}} \ar[d]&\\
& 0& 0.&
}
$$
 
The assumption $\mathcal J(\varphi|_{F_k})=\mathcal 
J(\varphi)|_{F_k}$ implies that 
$\mathcal J(\varphi|_{F_{k-1}})|_{F_k}
=\mathcal J(\varphi)|_{F_k}$ holds by 
$\mathcal J(\varphi|_{F_k})\subset 
\mathcal J(\varphi|_{F_{k-1}})|_{F_k}\subset 
\cdots \subset \mathcal J(\varphi)|_{F_k}$ in Lemma 
\ref{x-lem3.1}. 
If $\mathcal J(\varphi|_{F_i})|_{F_{i+1}}=
\mathcal J(\varphi)|_{F_{i+1}}$ on a neighborhood of $F_k$ 
in $F_{i+1}$, 
then $c_i$ is surjective on a neighborhood of $F_k$ 
in $F_{i+1}$ by the definition of $\mathcal J(\varphi|_{F_i})
|_{F_{i+1}}$. 
Then $d_i$ is also surjective on a neighborhood of $F_k$ 
in $F_{i}$ by the above big commutative diagram. 
By Nakayama's lemma, $\Coker b_i$ is zero on a neighborhood of $F_k$ 
in $F_{i}$. 
This implies that $\mathcal J(\varphi|_{F_i})=\mathcal J(\varphi)|_{F_i}$ 
on a neighborhood of $F_k$ 
in $F_i$. 
Thus, we obtain that 
$\mathcal J(\varphi|_{F_{i-1}})|_{F_i}=\mathcal J(\varphi)|_{F_i}$ 
on a neighborhood of $F_k$ 
in $F_i$ since we have 
$\mathcal J(\varphi|_{F_i})\subset 
\mathcal J(\varphi|_{F_{i-1}})|_{F_{i}}\subset 
\mathcal J(\varphi)|_{F_i}$ by Lemma \ref{x-lem3.1}. 
By repeating this argument, we see that 
$\mathcal J(\varphi|_{F_j})=\mathcal J(\varphi)|_{F_j}$ on a neighborhood 
of $F_k$ 
in $F_{j}$ 
for every $j$. 
This is the desired property. 
\end{proof}

\begin{lem}\label{x-lem3.4}
Assume that $\{H_1, \cdots, H_m\}$ satisfies condition $\spadesuit$ 
in Lemma \ref{x-lem3.2}. 
Let $H_{m+1}$ be a smooth Cartier divisor on $X$ such that 
$\sum _{i=1}^{m+1}H_i$ is a simple normal crossing divisor on $X$ and 
that $H_{m+1}$ contains no associated primes of 
\begin{equation*}
\mathcal O_X/\mathcal J(\varphi)  
\quad\text{and}\quad \mathcal O_{H_{i_1}\cap \cdots \cap H_{i_l}}/ 
\mathcal J(\varphi)|_{H_{i_1}\cap \cdots \cap H_{i_l}}
\end{equation*} for 
every $1\leq i_1< \cdots < i_l \leq m$. 
Then $\{H_1, \cdots, H_{m}, H_{m+1}\}$ also satisfies condition $\spadesuit$. 
\end{lem}

\begin{proof}
This is obvious from Remark \ref{x-rem3.3} (3). 
\end{proof}

\begin{lem}\label{x-lem3.5}
Let $\Lambda_0$ be a sublinear system of 
a free linear system $\Lambda$
on $X$ with 
$\dim \Lambda_0\geq 1$. 
Assume that $\{H_1, \cdots, H_m\}$ satisfies condition 
$\spadesuit$ in Lemma \ref{x-lem3.2}. 
We put 
\begin{equation*}
\mathcal F_0:=\{D\in \Lambda_0 \, | \, \{H_1, \cdots, H_m, D\} 
\ \text{satisfies}\  \spadesuit\}. 
\end{equation*}
Then 
$\mathcal F_0$ 
is Zariski open in $\Lambda_0$. 
In particular, if 
$\mathcal F_0$ 
is not empty, then it is a dense 
Zariski open set of $\Lambda_0$. 

Moreover, we assume that 
there exists $D_0\in \mathcal F_0$ such that 
$\mathcal J(\varphi|_V)=\mathcal J(\varphi)|_V$, where 
$V$ is an irreducible component of $H_1\cap \cdots 
\cap H_m \cap D_0$. 
Let $D$ be a member of 
$\mathcal F_0$ such that 
$V$ is an irreducible component of $H_1\cap \cdots \cap H_m \cap 
D$. 
Then $\mathcal J(\varphi|_D)=\mathcal J(\varphi)|_D$ holds on a neighborhood 
of $V$ 
in $D$. 
\end{lem}

\begin{proof}
We put 
$$
\mathcal F:=\{ D \in \Lambda \, |\, \{H_1, \cdots, H_m, D\} 
\ \text{satisfies}\  \spadesuit\}. 
$$ 
Then by Remark \ref{x-rem3.3} (3) 
and Lemma \ref{x-lem3.4}, it is easy 
to see that $\mathcal F$ is a dense Zariski open set in $\Lambda$ since 
$\Lambda$ is a free linear system on $X$. 
Therefore, $\mathcal F_0=\mathcal F\cap \Lambda_0$ is 
Zariski open in $\Lambda_0$. 
By Lemma \ref{x-lem3.2}, 
the equality $\mathcal J(\varphi|_D)=\mathcal J(\varphi)|_D$ holds 
on a neighborhood of $V$ 
in $D$ 
if $D \in \mathcal F_0$ and $V$ is an irreducible component 
of $H_1\cap \cdots \cap H_m \cap D$. 
We note that we do not need the compactness of $X$ 
in the proof of Lemma \ref{x-lem3.2}. 
Therefore, we can shrink $X$ and assume 
that $V=H_1\cap \cdots \cap H_m \cap D$ in the 
above argument. 
\end{proof}

The following theorem (see Theorem \ref{f-thm1.10}) 
is one of the key results of this paper. 

\begin{thm}[Density of good divisors:~Theorem \ref{f-thm1.10}]
\label{x-thm3.6} 
Let $X$ be a compact complex manifold, 
let $\Lambda$ be a free linear system on $X$ with $\dim \Lambda \geq 1$, 
and let $\varphi$ be a quasi-plurisubharmonic function on $X$. 
We put 
\begin{equation*}
\mathcal G :=\{ H\in \Lambda \, |\,  
{\text{$H$ is smooth and $\mathcal J(\varphi|_H)=
\mathcal J(\varphi)|_H$}}\}. 
\end{equation*}
Then $\mathcal G$ is dense in $\Lambda$ in the classical 
topology. 
\end{thm}
\begin{proof}
We may assume that $\varphi\not\equiv -\infty$. 
Throughout this proof, we put $f:=\Phi_\Lambda: X\to Y:=f(X)\subset \mathbb P^N$. 
Note that $N=\dim \Lambda$. 
We divide the proof into several steps. 

\setcounter{step}{-1}

\begin{step}[Idea of the proof]\label{x-step3.6.0}
In this step, we will explain an idea of the proof. 

A general member $H$ of $\Lambda$ is smooth by Bertini's theorem, 
and it always satisfies that 
$\mathcal J(\varphi|_H)\subset \mathcal J(\varphi)|_H$ 
by Lemma \ref{x-lem3.1}. 
Hence, the problem is to check 
that the opposite 
inclusion holds for any member of a dense subset in $\Lambda$. 

If $\dim \Lambda=1$, that is, $\Lambda$ is a pencil, 
then a member $H$ of $\Lambda$ is a fiber of 
the morphism $f=\Phi_\Lambda\colon X\to \mathbb P^1$
at a point $P \in \mathbb P^1$. 
By Fubini's theorem, we have $\mathcal J(\varphi|_{f^{-1}(P)})
\supset \mathcal J(\varphi)|_{f^{-1}(P)}$ for almost all $P\in \mathbb P^1$. 
This is the desired statement when $\dim \Lambda=1$. 
In general, we have $H_1\cap H_2\neq \emptyset$ for two 
general members $H_1$ and $H_2$ of $\Lambda$. 
For this reason, we choose $H_1$ and $H_2$ suitably (see 
Step \ref{x-step3.6.2} and 
Step \ref{x-step3.6.3}), take the blow-up 
$Z\to X$ along $H_1\cap H_2$, and reduce the problem to 
the pencil case (see Step \ref{x-step3.6.4}). 
\end{step}

\begin{step}\label{x-step3.6.1}
In this step, we will prove the theorem when $\dim Y=1$. 

Let $\psi_0, \ldots, \psi_N$ be a basis of 
$H^0(\mathbb P^N, \mathcal O_{\mathbb P^N}(1))$. 
We put 
$$
\mathscr Y =\{ (y, [a_0:\cdots : a_N])\in 
Y\times \mathbb P^N \, |\, a_0\psi_0(y)+
\cdots +a_N\psi_N(y)=0\}\subset Y\times \mathbb P^N
$$ 
and consider the following commutative diagram: 
$$
\xymatrix{
\mathscr X\ar[d]_{\widetilde f}
\ar@{^{(}->}[r]& X\times \mathbb P^N \ar[d]\ar[r]& X\ar[d]^-f\\
\mathscr Y \ar@{^{(}->}[r]\ar[dr]_-\pi& Y\times \mathbb P^N \ar[d]^-{p_2}\ar[r]
&Y \\ 
& \mathbb P^N , &
}
$$
where $\mathscr X\hookrightarrow X\times \mathbb P^N
\to X$ is the base change of $\mathscr Y\hookrightarrow Y
\times \mathbb P^N
\to Y$ by $f\colon X\to Y$, $p_2$ is the second projection, 
and 
$\pi=p_2|_{\mathscr Y}$. 
We can easily see that there exists a nonempty Zariski open set $U$ of 
$\mathbb P^N$ such that 
$\pi$ and $\widetilde f$ are \'etale and smooth over $U$, respectively. 
We note that $\Lambda=f^*|\mathcal O_{\mathbb P^N}(1)|$ by construction. 
Let $H$ be a member of $\Lambda$ corresponding to a point of $U$. 
Then $H$ is smooth and $\mathcal J(\varphi|_H)\subset 
\mathcal J(\varphi)|_H$ holds 
by Lemma \ref{x-lem3.1}. 
On the other hand, by applying Fubini's theorem to 
$(\pi\circ \widetilde f)^{-1}(U)\to U$, the opposite 
inclusion $\mathcal J(\varphi)|_H\subset \mathcal J(\varphi|_H)$ 
holds for almost all $H\in \Lambda$. 
This means that $\mathcal G$ is dense in $\Lambda$ in the classical topology. 
\end{step}

\begin{step}\label{x-step3.6.2}
In this step, we will prove the following preparatory lemma. 
\begin{lem}\label{x-lem3.7}
Let $D_1$ and $D_2$ be two members of $\Lambda$ such that 
$\{D_1, D_2\}$ satisfies condition $\spadesuit$ in Lemma \ref{x-lem3.2}. 
Let $\mathcal P_0$ 
be the pencil spanned by $D_1$ and $D_2$. 
Then, for almost all $D\in \mathcal P_0$, 
the member $D$ is smooth, 
$\{D\}$ satisfies condition $\spadesuit$, and 
$\mathcal J(\varphi|_D)=\mathcal J(\varphi)|_D$ holds 
outside $D_1\cap D_2$. 
\end{lem} 
\begin{proof}[Proof of Lemma \ref{x-lem3.7}] 
Let $A_{i}$ be a hyperplane in $\mathbb{P}^{N}$ such 
that $D_{i}=f^{*}A_{i}$, 
and $\rm{pr}\colon \mathbb{P}^{N} \dashrightarrow \mathbb{P}^{1}$ 
be the linear projection 
from the subspace $A_{1} \cap A_{2} \cong \mathbb{P}^{N-2}$. 
Then the meromorphic map 
$X\dashrightarrow \mathbb P^1$ associated with $\mathcal P_0$ 
is the composition of $f\colon X \to \mathbb{P}^{N}$ and 
$\rm{pr}\colon \mathbb{P}^{N} \dashrightarrow \mathbb{P}^{1}$. 
Since the blow-up of $\mathbb{P}^{N}$  along  $A_{1} \cap A_{2}$ 
gives an elimination of
the indeterminacy locus of $\rm{pr}\colon 
\mathbb{P}^{N} \dashrightarrow \mathbb{P}^{1}$, 
the blow-up $p\colon 
Z\to X$  along $D_1\cap D_2$ satisfies the following commutative diagram: 
\begin{equation*}
\xymatrix{
Z \ar[r]^{\quad p \quad}\ar[dr]_{ q }& X\ar@{-->}[d] \ar@{>}[r]^{f=\Phi_{\Lambda}} 
&\mathbb{P}^{N} \ar@{-->}[dl]^{{\rm{pr}} } \\
& \mathbb P^1. & 
}
\end{equation*}
By applying Fubini's theorem to $q\colon Z\to \mathbb P^1$, 
we obtain that 
$\mathcal J(p^*\varphi|_{q^{-1}(Q)})=\mathcal J(p^*\varphi)|_{q^{-1}
(Q)}$ for almost all $Q\in \mathbb P^1$. 
Lemma \ref{x-lem3.5} implies that $\{D\}$ satisfies condition $\spadesuit$ for 
almost all $D\in \mathcal P_0$. 
The desired properties follow 
since $p$ is an isomorphism outside 
$D_1\cap D_2$. 
\end{proof}
\end{step}

\begin{step}\label{x-step3.6.3} 
In this step, we will find a smooth member $H$ of $\Lambda$ 
such that $\mathcal J(\varphi|_H)=\mathcal J(\varphi)|_H$ and 
that $\{H\}$ satisfies condition $\spadesuit$. 

From now on, we assume that $\dim \Lambda\geq 2$ and that 
the statement of 
Theorem \ref{x-thm3.6} holds for lower dimensional free linear systems. 
We put $l:= \dim Y$. By Step \ref{x-step3.6.1}, we have a 
smooth member $H$ of $\Lambda$ with the desired 
properties when $l=1$. Therefore, we may assume that 
$l\geq 2$. 
We take two general hyperplanes $B_1$ and $B_2$ of $\mathbb P^N$. 
We put $D_1:=f^*B_1$ and $D_2:=f^*B_2$. 
By Lemma \ref{x-lem3.7}, 
we can take a hyperplane $A_1$ of $\mathbb P^N$ such that 
$X_1:=f^*A_1$ is smooth, 
$\{X_1\}$ satisfies condition $\spadesuit$, and $\mathcal J(\varphi|_{X_1})=
\mathcal J(\varphi)|_{X_1}$ outside $D_1\cap D_2$. 
Let $\Lambda|_{X_1}$ be the linear system on $X_1$ 
defined by $f_{1}\colon X_{1}=X \cap f^{-1}(A_{1})\to Y \cap A_{1} 
\subset  A_{1} \cong \mathbb P^{N-1}$, 
that is, the set of  pull-backs of the hyperplanes in 
$A_{1} \cong \mathbb P^{N-1}$ by $f_1$. 
By construction, we have $\dim \Lambda|_{X_1}=\dim \Lambda-1$. 
Thus, we see that 
\begin{equation*}
\{H\in \Lambda \, |\, X_1\cap H \ \text{is smooth and}\ 
\mathcal J(\varphi|_{X_1\cap H})=\mathcal J(\varphi|_{X_1})|_{X_1\cap H} 
\}
\end{equation*} 
is dense in $\Lambda$ in the classical topology by the induction 
hypothesis. 
Then we can take 
general hyperplanes $A_2, A_3, \cdots, A_l$ of $\mathbb P^N$ such that 
$\dim (A_1\cap \cdots\cap  A_l \cap Y)=0$ and that  
$f^{-1}(Q)$ is smooth and 
\begin{equation}\label{x-3.5}
\mathcal J(\varphi|_{f^{-1}(Q)})
=\mathcal J(\varphi|_{X_1})|_{f^{-1}(Q)}
\end{equation} for every $Q\in A_1\cap 
\cdots \cap A_l \cap Y$ by using the induction hypothesis repeatedly. 
Without loss of generality, we may assume that 
$f^{-1}(Q)\cap D_1\cap D_2=\emptyset$ for 
every $Q\in A_1\cap \cdots \cap A_l\cap Y$. 
Since 
\begin{equation*}
\mathcal J(\varphi|_{X_1})=\mathcal J(\varphi) |_{X_1}
\end{equation*} 
holds outside $D_1\cap D_2$, 
\begin{equation}\label{x-3.6}
\mathcal J(\varphi|_{X_1})|_{f^{-1}(Q)}=
\mathcal J(\varphi)|_{f^{-1}(Q)}
\end{equation} 
holds for every $Q\in A_1\cap \cdots \cap A_l\cap Y$. 
Therefore, we have 
\begin{equation*}
\mathcal J(\varphi|_{f^{-1}(Q)})
=\mathcal J(\varphi|_{X_1})|_{f^{-1}(Q)}=
\mathcal J(\varphi)|_{f^{-1}(Q)}
\end{equation*} for every $Q\in A_1\cap 
\cdots \cap A_l \cap Y$ by \eqref{x-3.5} and 
\eqref{x-3.6}. 
We may assume that $\{X_1=f^*A_1, f^*A_2, \cdots, 
f^*A_l\}$ satisfies condition $\spadesuit$. 
We take one point $P$ of $A_1\cap \cdots \cap A_l\cap Y$ and 
fix $A_2, \cdots, A_l$. 
By applying 
Lemma \ref{x-lem3.5} to the linear system 
\begin{equation*}
\Lambda_0 :=\{D\in \Lambda
\, | \, f^{-1}(P)\subset D\}, 
\end{equation*} 
we see that 
\begin{equation*} 
\mathcal F_0
:=\{ D\in \Lambda_0\, |\, \{D, f^*A_2, \cdots, 
f^*A_l\}\ \text{satisfies} \ \spadesuit\}
\end{equation*} 
is Zariski open in $\Lambda_0$. 
Note that $\mathcal F_0$ 
is nonempty by $X_1=f^*A_1\in 
\mathcal F_0$. 
By the latter conclusion of Lemma \ref{x-lem3.5}, 
we have: 
\begin{lem}\label{x-lem3.8}
Let $A_g$ be a general hyperplane of $\mathbb P^N$ passing 
through $P$. 
We put $X_g :=f^*A_g$. 
Then $\mathcal J(\varphi|_{X_g})=\mathcal J(\varphi)|_{X_g}$ holds 
on a neighborhood of $f^{-1}(P)$ 
in $X_g$. 
\end{lem} 
Let $\pi\colon X'\to X$ be 
the blow-up along $f^{-1}(P)$ and let 
${\rm{Bl}}_P(\mathbb P^N) \to \mathbb P^N$ be the blow-up 
of $\mathbb P^N$ at $P$. 
The induced morphism $\alpha\colon X' \to {\rm{Bl}}_P (\mathbb P^N) $ 
and the linear projection $\gamma\colon 
\mathbb P^N\dashrightarrow \mathbb P^{N-1}$ 
from $P\in \mathbb P^N$ satisfy the following commutative diagram. 
\begin{equation*}
\xymatrix{
X' \ar[r]^\pi\ar[dd]_\alpha& X \ar[d]^{f}\\ 
& Y \ar@{^{(}->}[d]\\ 
{\rm{Bl}}_P (\mathbb P^N) \ar[d]_{\beta}\ar[r]& \mathbb P^N\ar@{-->}[ld]
^\gamma\\ 
\mathbb P^{N-1}. &  
}
\end{equation*} 
We put $f':=\beta\circ \alpha$ and $Y':=f'(X')$. 
By applying the induction hypothesis to 
$f'\colon X'\to Y'\subset \mathbb P^{N-1}$, 
we can take a general hyperplane $A$ of $\mathbb P^{N-1}$ such that 
$f'^*A$ is smooth and that 
\begin{equation}\label{x-3.7}
\mathcal J( \pi^*\varphi|_{f'^{-1}(A)})=\mathcal J(\pi^*\varphi)|_{f'^{-1}(A)}. 
\end{equation}
Let $A_0$ be the hyperplane of $\mathbb P^N$ spanned by 
$P$ and $A$. 
Then we can see that 
\begin{equation}\label{x-3.8}
\{f^*A_2, \cdots, f^*A_l, H:=f^*A_0\}
\end{equation} satisfies condition $\spadesuit$ since 
$A$ is a general hyperplane of $\mathbb P^{N-1}$. 
We see that $\mathcal J(\varphi|_H)
=\mathcal J(\varphi)|_H$ by \eqref{x-3.7} and Lemma \ref{x-lem3.8}, 
and that $\{H\}$ satisfies condition $\spadesuit$ by \eqref{x-3.8}. 
Therefore this $H$ has the desired properties. 
\end{step}

\begin{step}\label{x-step3.6.4}
In this final step, we will prove that $\mathcal G$ is 
dense in $\Lambda$ in the classical topology. 

We will use the induction on $\dim X$. 
If $\dim X=1$, then $\dim Y=1$. 
Therefore, 
by Step \ref{x-step3.6.1}, 
we see that $\mathcal G$ is dense in $\Lambda$ in the classical 
topology. 
Therefore, we assume that 
$\dim X\geq 2$. 
If $\dim Y=1$, 
then $\mathcal G$ is dense by Step 
\ref{x-step3.6.1}. Thus, we may assume that 
$\dim \Lambda\geq \dim Y\geq 2$. 
By Step \ref{x-step3.6.3}, 
we can take a smooth member 
$H_0$ of $\Lambda$ such that 
$\mathcal J(\varphi|_{H_0})=\mathcal J(\varphi)|_{H_0}$ 
and that $\{H_0\}$ satisfies condition $\spadesuit$. 
By applying the induction hypothesis to $\Lambda|_{H_0}$, we 
see that 
\begin{equation*}
\mathcal G' :=\{ H' \in \Lambda \, |\, 
{\text{$H_0\cap H'$ is smooth and $\mathcal 
J(\varphi|_{H_0\cap H'})=\mathcal J(\varphi|_{H_0})
|_{H_0\cap H'}$}}\}
\end{equation*} 
is dense in $\Lambda$ in the classical topology. 
Since $\Lambda$ is a free linear system, we know 
that 
\begin{equation*}
\{H' \in \Lambda\, |\, \{H_0, H'\} \ \text{satisfies}\  \spadesuit\} 
\end{equation*}
is a nonempty Zariski open set in $\Lambda$. 
Therefore, 
\begin{equation*}
\mathcal G'':=\{H'\in \mathcal G' \, |\, \{H_0, H'\} \ \text{satisfies}
\ \spadesuit\} 
\end{equation*}
is also dense in $\Lambda$ in the classical topology. 
We note that 
\begin{equation*}
\mathcal J (\varphi|_{H_0\cap H'})
=\mathcal J(\varphi |_{H_0})|_{H_0\cap H'} 
=\mathcal J(\varphi)|_{H_0\cap H'}
\end{equation*} 
for every $H'\in \mathcal G'$ since $\mathcal J(\varphi|_{H_0})
=\mathcal J(\varphi)|_{H_0}$. 
Therefore, we obtain that 
\begin{equation}\label{x-3.9}
\mathcal J (\varphi|_{H_0\cap H'})
=\mathcal J(\varphi |_{H'})|_{H_0\cap H'} 
=\mathcal J(\varphi)|_{H_0\cap H'}
\end{equation} 
for every $H'\in \mathcal G''$. 
By the latter conclusion of Lemma \ref{x-lem3.5}, 
\eqref{x-3.9} indicates that $\mathcal J(\varphi|_{H'})=
\mathcal J(\varphi)|_{H'}$ on a neighborhood 
of $H_0\cap H'$ in $H'$ for every $H'\in \mathcal G''$. 
We consider the pencil $\mathcal P_{H'}$ spanned by 
$H_0$ and $H'\in \mathcal G''$, that is, 
the sublinear system of $\Lambda$ spanned by 
$H_0$ and $H'$. 
Let $D$ be a general member of $\mathcal P_{H'}$. 
Then by Lemma \ref{x-lem3.5}, $\{H_0, D\}$ satisfies $\spadesuit$ 
and $\mathcal J(\varphi|_D)=\mathcal J(\varphi)|_D$ holds on 
a neighborhood of $H_0\cap H'$ in $D$. 
Hence, by Lemma \ref{x-lem3.7}, we say that 
almost all members of $\mathcal P_{H'}$ are contained in $\mathcal G$. 
By this observation, we obtain that $\mathcal G$ is dense in 
$\Lambda$ in the classical topology. 
\end{step} 
Thus, we obtain the desired statement. 
\end{proof}

The following examples show that 
$\mathcal G$ in Theorem \ref{f-thm1.10} (Theorem \ref{x-thm3.6}) 
is not always Zariski open in $\Lambda$, 
or  even an intersection of countably many 
nonempty Zariski open sets of $\Lambda$

\begin{ex}\label{x-ex3.9}
We put 
\begin{equation*}
\psi(z):=\sum _{k=1}^\infty 2^{-k} \log 
\left| z-\frac{1}{k}
\right|
\end{equation*}
for $z\in \mathbb C$. 
Then it is easy to see that $\psi(z)$ is smooth for $|z|\geq 2$. 
By using a suitable 
partition of unity, we can construct a function $\varphi(z)$ on 
$\mathbb P^1$ such that 
$\varphi(z)=\psi(z)$ for $|z|\leq 3$ and that 
$\varphi(z)$ is smooth for $|z|\geq 2$ on $\mathbb P^1$.  
We can see that $\varphi$ is a quasi-plurisubharmonic function 
on $\mathbb P^1$. 
Since the Lelong number $\nu (\varphi, 1/n)$ of 
$\varphi$ at $1/n$ is $2^{-n}$ for every 
positive integer $n$, 
we see that $\mathcal J(\varphi)=\mathcal O_{\mathbb P^1}$ by Skoda's 
theorem (see, for example, 
\cite[(5.6) Lemma]{demailly}). 
Therefore $\mathcal J(\varphi)|_P=\mathcal O_P$ for 
every $P\in \mathbb P^1$. 
On the other hand, 
we have $\varphi(1/n)=-\infty$ for every positive integer $n$. 
If $P=1/n$ for some positive integer $n$, 
then $\mathcal J(\varphi|_P)=0$. 
Thus  
\begin{equation*}
\mathcal G:=\{ H \in |\mathcal O_{\mathbb P^1}(1)| \, | 
\, \mathcal J(\varphi|_H)=\mathcal J(\varphi)|_H\}
\end{equation*} 
is not a Zariski open set of $|\mathcal O_{\mathbb P^1}(1)|$ 
($\simeq \mathbb P^1$). 
\end{ex}

\begin{ex}\label{x-ex3.10}
We put $K:=\{ z \in \mathbb C \, |\, |z|\leq 1 \}$. 
Let $\{w_n \}_{n= 1}^{\infty}$ be a countable 
dense subset of $K$ and let $\{a_n\} _{n=1}^{\infty}$ be 
positive real numbers such that 
$\sum_{n=1}^{\infty} a_n <\infty$. 
We put 
\begin{equation*}
\psi(z):=\sum _{n=1}^{\infty} a_n \log |z-w_n|
\end{equation*} 
for $z\in \mathbb C$. 
Then we see that 
\begin{itemize}
\item[$\bullet$] $\psi$ is subharmonic on $\mathbb C$ and $\psi 
\not \equiv -\infty$, 
\item[$\bullet$] $\psi=-\infty$ on an uncountable dense subset of $K$, and 
\item[$\bullet$] $\psi$ is discontinuous almost everywhere on $K$. 
\end{itemize}
For the details, see \cite[Theorem 2.5.4]{ransford}. 
By using a suitable partition of unity, we can construct a function 
$\varphi(z)$ on $\mathbb P^1$ such that 
$\varphi(z)=\psi(z)$ for $|z|\leq 3$ and that 
$\varphi(z)$ is smooth for $|z|\geq 2$ on $\mathbb P^1$. 
Then we can see that $\varphi$ is a quasi-plurisubharmonic function 
on $\mathbb P^1$. 
In this case, 
\begin{equation*}
\mathcal G:=\{ H \in |\mathcal O_{\mathbb P^1}(1)| \, | 
\, \mathcal J(\varphi|_H)=\mathcal J(\varphi)|_H\}
\end{equation*} 
can not be written as an intersection of countably 
many nonempty Zariski open sets of $|\mathcal O_{\mathbb P^1}
(1)|$. 
\end{ex} 

As a direct consequence of Theorem \ref{x-thm3.6}, we have: 

\begin{cor}[Generic restriction theorem]\label{x-cor3.11} 
Let $X$ be a compact complex manifold and let 
$\varphi$ be a quasi-plurisubharmonic function on $X$. 
Let $\Lambda$ be a free linear system on $X$ with $\dim \Lambda
\geq 1$. 
We put 
\begin{equation*}
\mathcal H:=\{ H\in \mathcal G\, | \, H \ 
\text{contains no associated primes of}\ \mathcal O_X/\mathcal J(\varphi)\}, 
\end{equation*} 
where 
\begin{equation*}
\mathcal G :=\{ H\in \Lambda \, |\,  
{\text{$H$ is smooth and $\mathcal J(\varphi|_H)=
\mathcal J(\varphi)|_H$}}\}
\end{equation*} 
as in Theorem \ref{x-thm3.6}. 
Then $\mathcal H$ is dense in $\Lambda$ in the classical topology. 
Moreover, the following short sequence 
\begin{equation}\label{x-3.10}
0\to \mathcal J(\varphi)\otimes \mathcal O_X(-H)
\to \mathcal J(\varphi) \to \mathcal J(\varphi|_H)\to 0
\end{equation}  
is exact for any member $H$ of $\mathcal H$. 
\end{cor}
\begin{proof}
It is easy to see that 
\begin{equation*}
\{H\in \Lambda \, |\, H \ \text{contains no 
associated primes of} \ \mathcal O_X/\mathcal J(\varphi)\}
\end{equation*} 
is a nonempty 
Zariski open set of $\Lambda$ since 
$\Lambda$ is a free linear system 
on $X$. 
Therefore $\mathcal H$ is dense in $\Lambda$ in the classical 
topology by Theorem \ref{x-thm3.6} (see Theorem \ref{f-thm1.10}). 

Let $H$ be a member of $\mathcal H$. 
Then we obtain the 
following commutative diagram (see also \eqref{x-3.2}). 
$$
\xymatrix{
0 \ar[r]& \mathcal J(\varphi) \otimes \mathcal O_X(-H)
\ar@{^{(}->}[d]
\ar[r]^{\quad \quad \alpha}
&\mathcal J(\varphi) \ar@{^{(}->}[d]\ar[r]& \Coker\alpha 
\ar@{^{(}->}[d]\ar[r]& 0 \\ 
0 \ar[r]& \mathcal O_X(-H) \ar[r]& \mathcal O_X \ar[r]& \mathcal O_H 
\ar[r]& 0
}
$$
As in the proof of Lemma \ref{x-lem3.2}, 
we obtain $\Coker \alpha= \mathcal J(\varphi)|_H$. 
Since $H\in \mathcal H\subset \mathcal G$, 
we have $\mathcal J(\varphi)|_H=\mathcal J(\varphi|_H)$. 
Therefore, we obtain the desired short exact sequence 
\eqref{x-3.10}. 
\end{proof}

We will use Corollary \ref{x-cor3.11} in Step \ref{x-step1.9.3} 
in the proof of Proposition \ref{f-prop1.9} (see Section \ref{x-sec4}). 
We close this section with a remark on the multiplier ideal 
sheaves associated with effective $\mathbb Q$-divisors on 
smooth projective varieties. 

\begin{rem}[Multiplier ideal sheaves for 
effective $\mathbb Q$-divisors]\label{x-rem3.12}
Let $X$ be a smooth projective variety and let $D$ be an effective 
$\mathbb Q$-divisor on $X$. 
Let $S$ be a smooth hypersurface in $X$. 
We assume that $S$ is not contained in any component 
of $D$. 
Then we obtain the following short exact sequence: 
\begin{equation}\label{x-3.11}
0\to \mathcal J(X, D)\otimes \mathcal O_X(-S)
\to \mathrm{Adj}_S(X, D)\to \mathcal J(S, D|_S)\to 0, 
\end{equation}
where $\mathcal J(X, D)$ (resp.~$\mathcal J(S, D|_S)$) 
is the multiplier ideal sheaf associated with $D$ (resp.~$D|_S$). 
Note that $\mathrm{Adj}_S(X, D)$ is the adjoint ideal of $D$ along 
$S$ (see, for example, \cite[Theorem 3.3]{lazarsfeld-survey}). 
If $S$ is in general position with respect to $D$, 
then we can easily see that $\mathrm{Adj}_S(X, D)$ 
coincides with $\mathcal J(X, D)$. 
Let $H$ be a general member of a free linear system $\Lambda$ with 
$\dim \Lambda\geq 1$. 
Then we can easily see that 
\begin{equation}\label{x-3.12}
\mathcal J(H, D|_H)=\mathcal J(X, D)|_H
\end{equation}
holds by the definition of the multiplier ideal 
sheaves for effective $\mathbb Q$-divisors (see, for example, 
\cite[Example 9.5.9]{lazarsfeld-book}). 

By this observation, 
if $X$ is a smooth projective variety and 
$\varphi$ is a quasi-plurisubharmonic function 
associated with an effective $\mathbb Q$-divisor $D$ on $X$,  
then $\mathcal G$ in Theorem \ref{x-thm3.6} (see Theorem \ref{f-thm1.10}) 
and $\mathcal H$ in Corollary \ref{x-cor3.11} are dense  
Zariski open in $\Lambda$ by \eqref{x-3.12}. 
Moreover, we can easily check 
that \eqref{x-3.10} in Corollary 
\ref{x-cor3.11} holds for general members $H$ of $\Lambda$ by 
\eqref{x-3.11}. 
\end{rem}

\section{Proof of Proposition \ref{f-prop1.9}}\label{x-sec4}

In this section, we prove Proposition \ref{f-prop1.9} 
and explain how to reduce Corollary \ref{f-cor1.7} and 
Theorem \ref{f-thm1.4} to Theorem \ref{f-thmD} and Theorem \ref{f-thmA},  
respectively. 

\setcounter{step}{0}
\begin{proof}[Proof of Proposition \ref{f-prop1.9}] 
Our proof of Proposition \ref{f-prop1.9} consists of the following six steps:  
\begin{step}[Theorem \ref{f-thmA} $\Longrightarrow$ 
Theorem \ref{f-thmB}]\label{x-step1.9.1}
Since $N_1$ is semiample, 
we can take a smooth Hermitian metric 
$h_1$ on $N_1$ such that $\sqrt{-1}\Theta_{h_1}(N_1)
\geq 0$. 
We put $h_2:=h_1^{b/a}$. 
Then $$
\sqrt{-1}(\Theta_{hh_1}(F\otimes N_1)
-t \Theta_{h_2}(N_2))\geq 0
$$ for $0<t \ll 1$. 
It follows that $\mathcal J(hh_1)=\mathcal J (h)$ since $h_1$ is smooth. 
Therefore, by Theorem \ref{f-thmA}, 
we obtain the injectivity in Theorem \ref{f-thmB}. 
\end{step}

\begin{step}[Theorem \ref{f-thmB} $\Longrightarrow$ 
Theorem \ref{f-thmC}]\label{x-step1.9.2}
We assume that $R^if_*(K_X\otimes F\otimes \mathcal J(h))$ 
has a torsion subsheaf. 
Then we can find a very ample line bundle $H$ 
on $Y$ and $0\ne t\in H^0(Y, H)$ such that 
\begin{equation*}
\alpha\colon R^if_*(K_X\otimes F\otimes \mathcal J(h))
\to R^if_*(K_X\otimes F\otimes \mathcal J(h))\otimes H
\end{equation*}
induced by $\otimes t$ is not injective. 
We take a sufficiently large positive integer $m$ such that 
$\Ker \alpha\otimes H^{\otimes m}$ is generated 
by global sections. Then we have $H^0(Y, \Ker\alpha\otimes 
H^{\otimes m})\ne 0$. 
Without loss of generality, by making $m$ sufficiently large, we 
may further assume that 
\begin{equation}\label{x-eq4.1}
H^p(Y, R^qf_*(K_X\otimes F\otimes \mathcal J(h))\otimes 
H^{\otimes m})=0
\end{equation} and 
\begin{equation}\label{x-eq4.2}
H^p(Y, R^qf_*(K_X\otimes F\otimes \mathcal J(h))\otimes 
H^{\otimes m+1})=0
\end{equation}
for every $p>0$ and $q$ by the Serre vanishing theorem. 
By construction, 
\begin{equation}\label{x-eq4.3}
\begin{split}
H^0(Y, R^if_*(K_X\otimes F\otimes \mathcal J(h))\otimes 
H^{\otimes m})
\to 
H^0(Y, R^if_*(K_X\otimes F\otimes \mathcal J(h))\otimes H^{\otimes 
m+1})
\end{split}
\end{equation}
induced by $\alpha$ is not injective. 
Thus, by \eqref{x-eq4.1}, \eqref{x-eq4.2}, and \eqref{x-eq4.3},  
we see that 
\begin{equation*}
\begin{split}
H^i(X, K_X\otimes F\otimes \mathcal J(h)\otimes 
f^*H^{\otimes m}) 
\to  H^i(X, K_X\otimes F\otimes \mathcal J(h)\otimes 
f^*H^{\otimes m+1})
\end{split} 
\end{equation*}
induced by $\otimes f^*t$ is not injective. 
This contradicts Theorem \ref{f-thmB}. 
Therefore $R^if_*(K_X\otimes F\otimes \mathcal J(h))$ 
is torsion-free. 
\end{step}

\begin{step}[Theorem \ref{f-thmB} $\Longrightarrow$ 
Theorem \ref{f-thmD}]\label{x-step1.9.3}
We use the induction on $\dim Y$. 
If $\dim Y=0$, then the statement is obvious. 
We take a sufficiently large positive integer $m$ and a general 
divisor $B\in |H^{\otimes m}|$ such that 
$D:=f^{-1}(B)$ is smooth, 
contains no associated primes of $\mathcal O_X/\mathcal J(h)$, 
and satisfies $\mathcal J(h|_D)=\mathcal 
J(h)|_D$ by Theorem \ref{x-thm3.6} (see Theorem \ref{f-thm1.10}) and 
Corollary \ref{x-cor3.11}. 
By the Serre vanishing theorem, we may further assume that 
\begin{equation}\label{x-eq4.4}
H^i(Y, R^jf_*(K_X\otimes F\otimes \mathcal J(h)\otimes N)
\otimes H^{\otimes m})=0
\end{equation} 
for every $i>0$ and $j$. 
By Corollary \ref{x-cor3.11} and adjunction, 
we have the following short exact sequence:  
\begin{equation}\label{x-eq4.5}
\begin{split}
0&\to K_X\otimes F\otimes \mathcal J(h)\otimes 
N
\to K_X\otimes F\otimes \mathcal J(h) \otimes 
N\otimes f^*H^{\otimes m}
\\& \to K_D\otimes F|_D\otimes \mathcal J(h|_D)\otimes N|_D
\to 0. 
\end{split}
\end{equation}
Since $B$ is a general member of $|H^{\otimes m}|$, 
we may assume that $B$ contains no associated primes 
of $R^jf_*(K_X\otimes F\otimes \mathcal J(h)\otimes N)$ for 
every $j$. Hence, by \eqref{x-eq4.5}, we can obtain 
\begin{equation*}
\begin{split}
0 &\to R^jf_*(K_X\otimes F\otimes \mathcal J(h)\otimes 
N) 
\to R^jf_*(K_X\otimes F\otimes \mathcal J(h)\otimes 
N)\otimes H^{\otimes m}\\& 
\to R^jf_*(K_D\otimes F|_D\otimes \mathcal J(h|_D)
\otimes N|_D)\to 0
\end{split} 
\end{equation*}
for every $j$. 
By using the long exact sequence and the induction 
on $\dim Y$, 
we obtain 
\begin{equation*}
\begin{split}
H^i(Y, R^jf_*(K_X\otimes F\otimes \mathcal J(h)\otimes N))
=H^i(Y, R^jf_*(K_X\otimes F\otimes \mathcal J(h)\otimes N)
\otimes H^{\otimes m})
\end{split}
\end{equation*} 
for every $i\geq 2$ and $j$. 
Thus we have 
\begin{equation}\label{x-eq4.6}
H^i(Y, R^jf_*(K_X\otimes F\otimes \mathcal J(h)\otimes N))=0
\end{equation}
for every $i\geq 2$ and $j$ by \eqref{x-eq4.4}. 
By Leray's spectral sequence, \eqref{x-eq4.4}, and \eqref{x-eq4.6}, 
we have the 
following commutative diagram: 
$$
\xymatrix{
H^1(Y, \mathcal S^j )\ar[d]_\alpha
\ar@{^{(}->}[r]& H^{j+1}(X, 
K_X\otimes F\otimes \mathcal J(h)\otimes N)
\ar@{^{(}->}[d]^{\beta}
\\ 
H^1(Y, \mathcal S^j \otimes H^{\otimes m}) \ar@{^{(}->}[r]& H^{j+1}(X, 
K_X\otimes F\otimes \mathcal J(h)\otimes N\otimes 
f^*H^{\otimes m})
}
$$
for every $j$, where $\mathcal S^j$ stands for 
$R^jf_*(K_X\otimes F\otimes \mathcal J(h)\otimes N)$. 
Since $\beta$ is injective by Theorem \ref{f-thmB}, 
we obtain that $\alpha$ is also injective. 
By \eqref{x-eq4.4}, we have 
\begin{equation*}
H^1(Y, R^jf_*(K_X\otimes F\otimes \mathcal J(h)\otimes N)
\otimes H^{\otimes m})=0
\end{equation*} 
for every $j$. 
Therefore, we have $H^1(Y, 
R^jf_*(K_X
\otimes F\otimes \mathcal J(h)\otimes N))=0$ for every $j$. 
Thus, we obtain the desired vanishing theorem in Theorem \ref{f-thmD}. 
\end{step}

\begin{step}[Theorems \ref{f-thmC} and \ref{f-thmD} 
$\Longrightarrow$ 
Theorem \ref{f-thmB}]\label{x-step1.9.4}
By replacing $s$ and $N_2$ with $s^{\otimes m}$ 
and $N_2^{\otimes m}$ for some positive integer $m$ 
(see also Remark \ref{f-rem1.5}), 
we may assume that $N_2$ is globally generated. We consider 
\begin{equation*}
f:=\Phi_{|N_2|}\colon X\to Y. 
\end{equation*}
Then $N_2\simeq f^*H$ for some ample line bundle 
$H$ on $Y$ and  $s=f^*t$ for some $t\in H^0(Y, H)$. 
We take a smooth Hermitian metric $h_1$ on $N_1$ such that 
$\sqrt{-1}\Theta_{h_1}(N_1)\geq 0$. 
Then $\sqrt{-1}\Theta_{hh_1}(F\otimes N_1)\geq 0$ and 
$\mathcal J(hh_1)=\mathcal J(h)$. 
By Theorem \ref{f-thmC}, we obtain that 
\begin{equation*}
R^if_*(K_X\otimes F\otimes \mathcal J(h)\otimes N_1)
\end{equation*} 
is torsion-free for every $i$. 
Therefore, the map 
\begin{equation*}
R^if_*(K_X\otimes F\otimes \mathcal J(h)\otimes 
N_1)\to 
R^if_*(K_X\otimes F\otimes \mathcal J(h)\otimes N_1)\otimes H
\end{equation*} 
induced by $\otimes t$ is injective for every $i$. 
By $N_2\simeq f^*H$, we see that 
\begin{equation}\label{x-eq4.7}
\begin{split}
H^0(Y, R^if_*(K_X\otimes F\otimes \mathcal J(h)\otimes 
N_1))\to H^0(Y, 
R^if_*(K_X\otimes F\otimes \mathcal J(h)\otimes N_1\otimes N_2))
\end{split}
\end{equation} 
induced by $\otimes t$ 
is injective for every $i$.  
By Theorem \ref{f-thmD}, \eqref{x-eq4.7} implies that 
\begin{equation*}
\begin{split}
H^i(X, K_X\otimes F\otimes \mathcal J(h)\otimes N_1)
 \to H^i(X, K_X\otimes F\otimes 
\mathcal J(h)\otimes N_1\otimes N_2)
\end{split}
\end{equation*} 
induced by $\otimes s$ is injective for every $i$. 
\end{step}

\begin{step}[Theorem \ref{f-thmD} 
$\Longrightarrow$ 
Theorem \ref{f-thmE}]\label{x-step1.9.5}
The following lemma implies that 
$R^jf_*(K_X\otimes F\otimes \mathcal J(h))$ is 
a GV-sheaf by \cite[Theorem 25.5]{schnell} (see also 
\cite{hacon} and \cite{pp}).
For simplicity, we put 
$\mathcal F^j:=R^jf_*(K_X\otimes F\otimes \mathcal J(h))$ 
for every $j$. 
\begin{lem}\label{f-lem4.1}
For every finite \'etale morphism $p\colon B\to A$ of 
Abelian varieties and every ample line bundle $H$ on $B$, we have 
\begin{equation}\label{x-eq4.8}
H^i(B, H\otimes p^*\mathcal F^j)=0
\end{equation} 
for every $i>0$ and $j$. 
\end{lem}
\begin{proof}[Proof of Lemma \ref{f-lem4.1}]
We put $Z:=B\times _A X$. 
Then we have the following commutative diagram. 
\begin{equation}\label{x-eq4.9}
\xymatrix{
Z \ar[r]^q\ar[d]_g& X\ar[d]^f\\ 
B \ar[r]_p & A
}
\end{equation}
By construction, $q$ is also finite and \'etale. 
Therefore, we have $q^*K_X=K_Z$ and 
$q^*\mathcal J(h)=\mathcal J(q^*h)$. 
By the flat base change theorem, 
\begin{equation*}
p^*R^jf_*(K_X\otimes F\otimes \mathcal J(h))\simeq 
R^jg_*(K_Z\otimes q^*F 
\otimes \mathcal J(q^*h)). 
\end{equation*}
By Theorem \ref{f-thmD}, we obtain the desired vanishing 
\eqref{x-eq4.8}. 
\end{proof} 
\end{step}

\begin{step}[Theorems \ref{f-thmC} and \ref{f-thmE} 
$\Longrightarrow$ 
Theorem  \ref{f-thmF}]\label{x-step1.9.6}
By Theorem \ref{f-thmC}, we have 
$\mathcal F^j:=R^jf_*(K_X\otimes F\otimes \mathcal J(h))=0$ 
for $j>\dim X-\dim f(X)$. 
We consider the following spectral sequence: 
\begin{equation*}
E_2^{pq}=H^p(A, \mathcal F^q \otimes L)\Rightarrow 
H^{p+q}(X, K_X\otimes F\otimes \mathcal J(h)\otimes f^*L)
\end{equation*} 
for every $L\in \Pic ^0(A)$. 
Note that $\mathcal F^j$ is a GV-sheaf for every $j$ and 
that $\mathcal F^j=0$ for $j>\dim X-\dim f(X)$. 
Then we obtain 
\begin{equation*}
\mathrm{codim}_{\Pic^0(A)}\{L\in \Pic^0(A) \, 
|\, H^i(X, K_X\otimes F\otimes \mathcal J(h)\otimes 
f^*L)\ne 0\} 
\geq i -(\dim X -\dim f(X)) 
\end{equation*}
for every $i\geq 0$. 
\end{step}
We completed the proof of Proposition \ref{f-prop1.9}. 
\end{proof}

We prove Corollary \ref{f-cor1.7} as an application of 
Theorem \ref{f-thmD}. 

\begin{proof}[Proof of Corollary \ref{f-cor1.7}
{\textit{(Theorem \ref{f-thmD} $\Longrightarrow$ Corollary \ref{f-cor1.7})}}]
By Theorem \ref{f-thmD}, 
we have 
\begin{equation*}
H^p(Y, R^if_*(K_X\otimes F\otimes \mathcal J(h))\otimes 
H^{\otimes m-p})=0
\end{equation*}
for every $p\geq 1$, $i\geq 0$, and $m\geq 
\dim Y+1$. Thus, the Castelnuovo--Mumford regularity 
(see \cite[Section 1.8]{lazarsfeld-book1}) 
implies that $R^if_*(K_X\otimes F\otimes \mathcal J(h))\otimes 
H^{\otimes m}$ is globally generated for every $i\geq 0$ and 
$m\geq \dim Y+1$. 
\end{proof}

We close this section with 
a proof of Theorem \ref{f-thm1.4} based on Theorem \ref{f-thmA} 
for the reader's convenience. 

\begin{proof}[Proof of Theorem \ref{f-thm1.4} 
{\textit{(Theorem \ref{f-thmA} $\Longrightarrow$ Theorem \ref{f-thm1.4})}}]  
Let $A$ be an ample line bundle on $V$. 
Then there exists a sufficiently large positive integer $m$ such that 
$A^{\otimes m}$ is very ample and that 
$H^i(V, K_V\otimes L\otimes \mathcal J(h_L)\otimes 
A^{\otimes m})=0$ for every $i>0$ by the Serre vanishing 
theorem. 
We can take a smooth Hermitian metric $h_A$ on $A$ such that 
$\sqrt{-1}\Theta_{h_A}(A)$ is a smooth positive $(1, 1)$-form on $V$. 
Therefore, we have 
$\sqrt{-1}\Theta_{h^m_A}(A^{\otimes m})\geq 0$. 
By the condition $\sqrt{-1}\Theta_{h_L}(L)\geq \varepsilon \omega$, 
we see that $\sqrt{-1}(\Theta_{h_L}(L)-t\Theta_{h^m_A}
(A^{\otimes m}))\geq 0$ for some $0<t\ll 1$. 
We take a nonzero global section $s$ of $A^{\otimes m}$. 
By Theorem \ref{f-thmA}, we see that 
\begin{equation*}
\times s\colon H^i(V, K_V\otimes L\otimes \mathcal J(h_L))
\to H^i(V, K_V\otimes L\otimes \mathcal J(h_L)\otimes 
A^{\otimes m})
\end{equation*}
is injective for every $i$. 
Thus, we obtain that $H^i(V, K_V\otimes L\otimes 
\mathcal J(h_L))=0$ for every $i>0$. 
\end{proof}

\section{Proof of Theorem \ref{f-thmA}}\label{f-sec5} 

In this section, we will give the proof of Theorem \ref{f-thmA}. 

\begin{thm}[Theorem \ref{f-thmA}]\label{f-thm5.1}
Let $F$ $($resp.~$M$$)$ 
be a line bundle on a compact K\"ahler manifold $X$ 
with a singular Hermitian 
metric $h$ $($resp.~a smooth Hermitian metric $h_M$$)$ satisfying 
\begin{align*}
\sqrt{-1}\Theta_{h_M} (M) \geq 0 \text{ and }
\sqrt{-1}\Theta_{h}(F) - b \sqrt{-1}\Theta_{h_M} (M) \geq 0 
\text{ for some $b>0$}. 
\end{align*}
Then for a $($nonzero$)$ section $s \in H^{0}(X, M)$, 
the multiplication map induced by $\otimes s$ 
\begin{equation*}
\times s\colon H^{q}(X, K_X \otimes F \otimes \mathcal J(h)) 
\xrightarrow{\quad \otimes s \quad } 
H^{q}(X, K_X \otimes F \otimes \mathcal J(h) \otimes M )
\end{equation*}
is injective for every $q$. 
Here $K_X$ is the canonical bundle of $X$ 
and $\mathcal J(h)$ is the multiplier ideal sheaf of $h$. 
\end{thm}

\begin{proof}[Proof of Theorem \ref{f-thm5.1} $($Theorem \ref{f-thmA}$)$] 
The proof can be divided into four steps. 
\setcounter{step}{0}
\begin{step}\label{f-st1}
Throughout the proof, we fix a K\"ahler form  $\omega$ on $X$. 
For a given singular Hermitian metric $h$ on $F$, 
by applying \cite[Theorem 2.3]{dps} to the weight of $h$,  
we obtain a family of singular Hermitian metrics 
$\{h_{\e} \}_{1\gg \e>0}$ on $F$ with 
the following properties: 
\begin{itemize}
\item[(a)]$h_{\e}$ is smooth on $Y_{\e}:=X \setminus Z_{\e}$, 
where $Z_{\e}$ is a proper closed analytic subset on $X$. 
\item[(b)]$h_{\e'} \leq h_{\e''} \leq h$ holds on $X$ 
when $\e' > \e'' > 0$.
\item[(c)]$\I{h}= \I{h_{\e}}$ on $X$.
\item[(d)]$\sqrt{-1} \Theta_{h_{\e}}(F) 
\geq b\sqrt{-1}\Theta_{h_{M}}(M) -\e \omega$ on $X$. 
\end{itemize}
Here property (d) is obtained from the assumption 
$\sqrt{-1}\Theta_{h}(F) \geq b \sqrt{-1}\Theta_{h_M} (M)$. 

The main difficulty of the proof is that $Z_{\e}$ may essentially depend on $\e$, 
compared to \cite{matsumura4} in which $Z_{\e}$ is independent of $\e$. 
To overcome this difficulty, 
we consider suitable complete K\"ahler forms 
$\{\omega_{\e, \delta}\}_{\delta>0}$ on $Y_{\e}$ 
such that $\omega_{\e, \delta}$ converges to 
$\omega$ as $\delta \to 0$. 
To construct such complete K\"ahler forms, 
we first take a complete K\"ahler form  $\omega_{\e}$ on $Y_{\e}$ 
with the following properties: 
\begin{itemize}
\item[$\bullet$] $\omega_{\e}$ is a complete K\"ahler form on 
$Y_{\e}$.
\item[$\bullet$] $\omega_{\e} \geq \omega $ on $Y_{\e}$. 
\item[$\bullet$] $\omega_{\e}=\deldel \Psi_{\e}$ 
for some bounded function $\Psi_{\e}$ 
on a neighborhood of every $p \in X$.    
\end{itemize} 
See \cite[Section 3]{fujino-osaka} for the construction of $\omega_{\e}$. 
For the K\"ahler form $\omega_{\e, \delta}$ on $Y_{\e}$ defined to be 
$$
\omega_{\e, \delta}:=\omega + \delta \omega_{\e} 
\text{ for } \e \text{ and } \delta  \text{ with } 0<\delta \ll \e, 
$$
it is easy to see the following properties hold: 
\begin{itemize}
\item[(A)] $\omega_{\e, \delta}$ is a 
complete K\"ahler form on $Y_{\e}=X\setminus Z_{\e}$ for every 
$\delta>0$.
\item[(B)] $\omega_{\e, \delta} \geq \omega $ on $Y_{\e}$ 
for every  $\delta>0$. 
\item[(C)] $\Psi + \delta \Psi_{\e}$ is a bounded local potential function of 
$\omega_{\e, \delta}$ and converges to $\Psi$ as $\delta \to 0$.  
\end{itemize} 
Here $\Psi$ is a local potential function of $\omega$. 
The first property enables us to consider harmonic forms 
on the noncompact $Y_{\e}$, 
and the third property enables us to construct the de Rham--Weil 
isomorphism 
from the $\dbar$-cohomology on $Y_{\e}$ to 
the $\rm{\check{C}}$ech cohomology on $X$. 

\begin{rem}\label{f-rem5.2} 
In the proof of Theorem \ref{f-thm5.1}, 
we actually consider only a countable sequence $\{\e_{k}\}_{k=1}^{\infty}$ 
(resp.~$\{\delta_{\ell}\}_{\ell=1}^{\infty}$) conversing to zero since 
we need to apply Cantor's diagonal argument, 
but we often use the notation $\e$ (resp.~$\delta$) for simplicity. 
\end{rem}

For the proof, it is sufficient to show that 
an arbitrary cohomology class 
$\eta \in H^{q}(X, K_X\otimes F \otimes \I{h})$ 
satisfying $s \eta = 0 \in H^{q}(X, K_X\otimes F \otimes \I{h} \otimes M)$ 
is actually zero. 
We represent the cohomology 
class $\eta \in H^{q}(X, K_X\otimes F \otimes \I{h})$ 
by a $\dbar$-closed $F$-valued $(n,q)$-form $u$ 
with $\| u \|_{h, \omega} < \infty$ 
by using the standard de Rham--Weil isomorphism 
\begin{equation*}
H^{q}(X, K_X\otimes F \otimes \I{h})
\cong 
\frac{\Ker\dbar\colon L^{n,q}_{(2)}(F)_{h,\omega}
\to L^{n,q+1}_{(2)}(F)_{h,\omega}}
{\Image\dbar\colon L^{n,q-1}_{(2)}(F)_{h,\omega}
\to L^{n,q}_{(2)}(F)_{h,\omega}}. 
\end{equation*}
Here $\dbar$ is the densely defined closed operator 
defined by the usual $\dbar$-operator 
and $L^{n,q}_{(2)}(F)_{h,\omega}$ is   
the $L^2$-space of $F$-valued $(n,q)$-forms on $X$ 
with respect to the $L^2$-norm $\|\bullet \|_{h, \omega}$ defined by 
$$
\|\bullet \|^2_{h, \omega}:= \int_{X} |\bullet |^2_{h, \omega}\, dV_{\omega}, 
$$
where $dV_{\omega}:=\omega^n /n!$ and $n:=\dim X$. 
Our purpose is 
to prove that $u$ is $\dbar$-exact 
(namely, $u \in \Image\dbar \subset L^{n,q}_{(2)}(F)_{h,\omega}$) 
under the assumption that the cohomology class of $s u$ is zero in 
$H^{q}(X, K_X\otimes F \otimes \I{h} \otimes M)$. 

From now on, we mainly consider the $L^{2}$-space 
$L^{n,q}_{(2)}(Y_{\e},F)_{h_{\e},\omega_{\e, \delta}}$ 
of $F$-valued $(n,q)$-forms on $Y_{\e}$ (not $X$) with respect to 
$h_{\e}$ and $\omega_{\e, \delta}$ (not $h$ and $\omega$). 
For simplicity we put
$$
L^{n,q}_{(2)}(F)_{\e, \delta}:= L^{n,q}_{(2)}(Y_{\e},F)_{h_{\e},\omega_{\e, \delta}}
\text{\quad and \quad}
\|\bullet\|_{\e, \delta}:=\|\bullet\|_{h_{\e}, \omega_{\e, \delta}}. 
$$
The following inequality plays an important role in the proof.  
\begin{align}\label{eq5.1}
\|u\|_{\e, \delta} \leq \|u\|_{h, \omega_{\e, \delta}} \leq \|u\|_{h, \omega} <\infty.  
\end{align}
In particular, the norm $\|u\|_{\e, \delta}$ is uniformly bounded 
since the right hand side is independent of $\e$, $\delta$. 
The first inequality follows from property (b) of $h_{\e}$, 
and the second inequality follows from 
Lemma \ref{f-lem2.4} and property (B) of $\omega_{\e, \delta}$. 
Strictly speaking, the left hand side should be 
$\|u|_{Y_{\e}}\|_{\e, \delta}$, 
but we often omit the symbol of restriction. 
Now we have the following orthogonal decomposition 
(for example see \cite[Proposition 5.8]{matsumura4}).  
$$
L^{n,q}_{(2)}(F)_{\e, \delta}= 
\Image\dbar \,  \oplus 
\mathcal{H}^{n,q}_{\e, \delta}(F)\, \oplus 
\Image \dbar^{*}_{\e, \delta}. 
$$
Here $\dbar^{*}_{\e, \delta}$ is (the maximal extension of) 
the formal adjoint of the $\dbar$-operator 
and $\mathcal{H}^{n,q}_{\e, \delta}(F)$ is the set of 
harmonic $F$-valued $(n,q)$-forms on $Y_{\e}$, namely    
$$
\mathcal{H}^{n,q}_{\e, \delta}(F):= 
\{ w \in L^{n,q}_{(2)}(F)_{\e, \delta} \, | \, 
\dbar w=0 \text{ and } \dbar^{*}_{\e, \delta}w=0 \}. 
$$

\begin{rem}\label{f-rem5.3}
The formal adjoint coincides with the Hilbert space adjoint 
since $\omega_{\e, \delta}$ is complete for 
$\delta >0$ (see, for example, \cite[(3.2) Theorem 
in Chapter  V\hspace{-.1em}I\hspace{-.1em}I\hspace{-.1em}I]{demailly-note}). 
The $\dbar$-operator also depends on $h_{\e}$ and $\omega_{\e, \delta}$ 
in the sense that the domain and range 
of the closed operator $\dbar$ depend on them, 
but we abbreviate $\dbar_{\e, \delta}$ to $\dbar$.
\end{rem}

The $F$-valued $(n,q)$-form $u$ (representing $\eta$) 
belongs to $L^{n,q}_{(2)}(F)_{\e,\delta}$ by \eqref{eq5.1}, 
and thus $u$ can be decomposed as follows:  
\begin{align}\label{eq5.2}
u=\dbar w_{\e, \delta}+u_{\e, \delta}\quad 
\text{for some } 
w_{\e, \delta} \in {\rm{Dom}\, \dbar} \subset 
L^{n,q-1}_{(2)}(F)_{\e,\delta}\text{ and } 
u_{\e, \delta} \in \mathcal{H}^{n,q}_{\e, \delta}(F).   
\end{align}
Note that the orthogonal projection of $u$ to 
${\Image \dbar^{*}_{\e, \delta}}$ 
must be zero since $u$ is $\dbar$-closed.
\end{step}

\begin{step}\label{f-st2}
The purpose of this step is to prove Proposition \ref{f-prop5.7}, 
which reduces the proof to the study of the asymptotic behavior 
of the norm of $su_{\e,\delta}$. 
When we consider a suitable limit of $u_{\e,\delta}$ 
in the following proposition,  
we need to carefully choose the $L^2$-space 
since the $L^2$-space $L^{n,q}_{(2)}(F)_{\e,\delta}$ 
depends on $\e$ and $\delta$. 
We remark that $\{\e\}_{\e>0}$ and $\{\delta\}_{\delta>0}$ 
denote countable sequences converging to zero (see Remark \ref{f-rem5.2}). 
Let $\{\delta_{0}\}_{\delta_{0}>0}$ 
denote another countable sequence  converging to zero. 

\begin{prop}\label{f-prop5.4}
There exist a subsequence $\{\delta_{\nu}\}_{\nu=1}^{\infty}$ of 
$\{\delta\}_{\delta>0}$ and 
$\alpha_{\e} \in L^{n,q}_{(2)}(F)_{h_{\e}, \omega}$ with 
the following properties: 
\begin{itemize}
\item[$\bullet$] For any $\e, \delta_{0}>0$, 
as $\delta_{\nu}$ tends to $0$, 
$$
u_{\e, \delta_{\nu}} \text{ converges to } \alpha_{\e} 
\text{ with respect to the weak } L^2\text{-topology in } 
L^{n,q}_{(2)}(F)_{\e,  \delta_{0}}. 
$$
\item[$\bullet$] For any $\e>0$, 
$$
\| \alpha_{\e} \|_{h_{\e},\omega}
\leq \varliminf_{ \delta_{0} \to 0}\| \alpha_{\e} \|_{\e, \delta_{0}}
\leq \varliminf_{\delta_{\nu} \to 0}\| u_{\e, \delta_{\nu}} \|_{\e,\delta_{\nu}}
\leq \|u\|_{h, \omega}. 
$$
\end{itemize}
\end{prop}

\begin{rem}\label{f-rem5.5} 
The weak limit $\alpha_{\e}$ does not depend on  $\delta_{0}$,  
and the subsequence $\{\delta_{\nu}\}_{\nu=1}^{\infty}$ 
does not depend on $\e$ and $\delta_{0}$. 
\end{rem}

\begin{proof}[Proof of Proposition \ref{f-prop5.4}]
For given $\e, \delta_{0}>0$, 
by taking a sufficiently small $\delta$ with $0<\delta<\delta_{0}$, 
we have 
\begin{align}\label{eq5.3}
\|u_{\e, \delta}\|_{\e, \delta_{0}}
\leq \|u_{\e, \delta}\|_{\e, \delta}
\leq \|u\|_{\e, \delta} 
\leq \|u\|_{h, \omega}. 
\end{align}
The first inequality follows from 
$\omega_{\e, \delta} \leq \omega_{\e, \delta_{0}}$ and Lemma \ref{f-lem2.4}, 
the second inequality follows since $u_{\e, \delta}$ is 
the orthogonal projection of $u$ with respect to $\e, \delta$, 
and the last inequality follows from \eqref{eq5.1}. 
Since the right hand side is independent of $\delta$, 
the family $\{u_{\e, \delta}\}_{\delta>0}$ is 
uniformly bounded in $L^{n,q}_{(2)}(F)_{\e,  \delta_{0}}$.
Therefore, there exists a subsequence $\{\delta_{\nu}\}_{\nu=1}^{\infty}$ of 
$\{\delta\}_{\delta>0}$ such that 
$u_{\e, \delta_{\nu}}$ converges to 
$ \alpha_{\e,\delta_{0}}$ 
with respect to the weak $L^2$-topology in $L^{n,q}_{(2)}(F)_{\e,  \delta_{0}}$ 
This subsequence 
$\{\delta_{\nu}\}_{\nu=1}^{\infty}$ may depend on $\e, \delta_{0}$, 
but we can choose a subsequence independent of them  
by applying Cantor's diagonal argument. 

Now we show that $\alpha_{\e, \delta_{0}}$ does not depend on $\delta_{0}$. 
For arbitrary $\delta'_0, \delta''_0$ 
with $0 < \delta'_0\leq \delta''_0$, 
the natural inclusion $L^{n,q}_{(2)}(F)_{\e,  \delta'_0}\to 
L^{n,q}_{(2)}(F)_{\e,  \delta''_0}$ is a bounded operator 
(continuous linear map) 
by $\|\bullet \|_{\e,  \delta''_0} \leq \|\bullet \|_{\e,  \delta'_0}$, 
and thus $u_{\e, \delta_{\nu}}$ weakly converges to $\alpha_{\e, \delta'_0}$ 
in not only $L^{n,q}_{(2)}(F)_{\e,  \delta'_0}$ but also 
$L^{n,q}_{(2)}(F)_{\e,  \delta''_0}$ by Lemma \ref{f-lem2.5}. 
Therefore, it follows that $\alpha_{\e,\delta'_0}=\alpha_{\e,\delta''_0}$
since the weak limit is unique. 

Finally, we consider the norm of $\alpha_{\e}$. 
It is easy to see that 
$$
\| \alpha_{\e} \|_{\e, \delta_{0}}
\leq \varliminf_{\delta_{\nu} \to 0}\| u_{\e, \delta_{\nu}} \|_{\e,\delta_{0}}
\leq \varliminf_{\delta_{\nu} \to 0}\| u_{\e, \delta_{\nu}} \|_{\e,\delta_{\nu}}
\leq \|u\|_{h, \omega}. 
$$
The first inequality follows since the norm is 
lower semicontinuous with respect to the weak convergence, 
the second inequality follows from 
$\omega_{\e, \delta_{0}} \geq \omega_{\e, \delta_{\nu}}$, 
and the last inequality follows from \eqref{eq5.3}. 
Fatou's lemma yields 
$$
\| \alpha_{\e} \|^2_{h_{\e}, \omega} 
=\int_{Y_{\e}} | \alpha_{\e} |^2_{h_{\e}, \omega}\,dV_{\omega}
\leq \varliminf_{\delta_{0} \to 0} 
\int_{Y_{\e}} | \alpha_{\e} |^2_{h_{\e}, 
\omega_{\e, \delta_{0}}}\,dV_{\omega_{\e, \delta_{0}}}
=\varliminf_{\delta_{0} \to 0} \| \alpha_{\e} \|^2_{\e, \delta_{0}}. 
$$
These inequalities lead to the desired estimate in the proposition. 
\end{proof}

For simplicity, 
we use the same notation $\{u_{\e, \delta}\}_{\delta>0}$ for 
the subsequence 
$\{u_{\e, \delta_{\nu}}\}_{\nu=1}^{\infty}$ in Proposition \ref{f-prop5.4}. 
We fix $\e_{0}>0$ and consider the weak limit of $\alpha_{\e}$ in 
the fixed $L^2$-space $L^{n,q}_{(2)}(F)_{h_{\e_{0}},\omega}$. 
For a sufficiently small $\e>0$, we have 
$$
\| \alpha_{\e} \|_{h_{\e_{0}}, \omega } 
\leq \| \alpha_{\e} \|_{h_{\e}, \omega }
\leq \|u\|_{h,\omega} 
$$
by property (b) and Proposition \ref{f-prop5.4}. 
By taking a subsequence of $\{\alpha_{\e} \}_{\e>0}$, 
we may assume that 
$\alpha_{\e} $ 
weakly converges to some $\alpha$ in 
$L^{n,q}_{(2)}(F)_{h_{\e_{0}},\omega}$. 

\begin{prop}\label{f-prop5.6}
If the weak limit $\alpha$ is zero in 
$L^{n,q}_{(2)}(F)_{h_{\e_{0}},\omega}$, 
then the cohomology class $\eta$ is zero in 
$H^{q}(X, K_X\otimes F \otimes \I{h})$. 
\end{prop}
\begin{proof}[Proof of Proposition \ref{f-prop5.6}]
For every $\delta$ with $0< \delta \leq \delta_{0}$, 
we can easily check 
\begin{align*}
u-u_{\e,\delta} \in \Image\dbar \text{ in } 
L^{n,q}_{(2)}(F)_{\e, \delta}
&\subset  \Image\dbar \text{ in } 
L^{n,q}_{(2)}(F)_{\e, \delta_{0}}
\end{align*} 
from the construction of $u_{\e,\delta}$. 
As $\delta \to 0$, we obtain 
\begin{align*}
u - \alpha_{\e} & \in \Image \dbar \text{ in }  
L^{n,q}_{(2)}(F)_{\e, \delta_{0}}  
\end{align*}
by Lemma \ref{f-lem2.6} and Proposition \ref{f-prop5.4}.  
We remark that $\Image\dbar$ is a closed subspace 
(see \cite[Proposition 5.8]{matsumura4}). 
On the other hand, 
we have the following commutative diagram: 
\[\xymatrix{
&\Ker \dbar \text{ in } L^{n,q}_{(2)}(F)_{\e, \delta_{0}}  
\ar[r]^{q_{1}\ \ }&  
\dfrac{\Ker \dbar}{\Image \dbar} 
\text{ of }  L^{n,q}_{(2)}(F)_{\e, \delta_{0}} 
\ar[r]^{\cong\ \ \ }_{f_{1}\ \ }
&\check{H}^{q}(X, K_X\otimes F \otimes \I{h})\\
&\Ker \dbar \text{ in } L^{n,q}_{(2)}(F)_{h_{\e}, \omega}
\ar[u]_{j_{1}} \ar[r]^{j_{2}\ \ } 
&\Ker \dbar \text{ in } L^{n,q}_{(2)}(F)_{h_{\e_{0}}, 
\omega} \ar[r]^{q_{2}\ \ }  
&\dfrac{\Ker \dbar}{\Image \dbar} 
\text{ of } L^{n,q}_{(2)}(F)_{h_{\e_{0}}, \omega}
\ar[u]^{\cong}_{f_2}.
}\]
Here $j_{1}$, $j_{2}$ are the natural inclusions, 
$q_{1}$, $q_{2}$ are the natural quotient maps, 
and $f_1$, $f_2$ are the de Rham--Weil isomorphisms 
(see \cite[Proposition 5.5]{matsumura4} for the construction). 
Strictly speaking, $f_{1}$ is an isomorphism to 
$\check{H}^{q}(X, K_X\otimes F \otimes \I{h_{\e}})$, 
but which coincides with $\check{H}^{q}(X, K_X\otimes F \otimes \I{h})$
by property (c). 
To check that $j_{2}$ is well-defined,  
we have to see that $\dbar w=0$ on $Y_{\e_{0}}$ if $\dbar w=0$ on $Y_{\e}$. 
By the $L^2$-integrability 
and \cite[(7.3) Lemma, 
Chapter V\hspace{-.1em}I\hspace{-.1em}I\hspace{-.1em}I]{demailly-note}, 
the equality $\dbar w=0$ can be 
extended from $Y_{\e}$ to $X$ (in particular $Y_{\e_{0}}$). 
The key point here is 
the $L^2$-integrability with respect to $\omega$ 
(not $\omega_{\e,\delta}$). 

Since $j_{2}(u-\alpha_{\e})$ weakly 
converges to $j_{2}(u-\alpha)$ 
and the $\dbar$-cohomology is finite dimensional, 
we obtain 
$$
\lim_{\e \to 0}q_{2}(u-\alpha_{\e})
=q_{2}(u-\alpha)=q_{2}(u)   
$$
by Lemma \ref{f-lem2.5} and the assumption $\alpha =0$.
On the other hand, it follows that 
$q_{1}(u-\alpha_{\e})=0$
from the first half argument. 
Hence, we have $q_{2}(u)=0$, that is, 
$u \in \Image \dbar \subset L^{n,q}_{(2)}(F)_{h_{\e_{0}}, \omega}$. 
From $q_{2}(u)=0$, 
we can prove the conclusion, 
that is, $u \in \Image \dbar \subset L^{n,q}_{(2)}(F)_{h, \omega}$. 
Indeed, we can obtain $q_{3}(u)=0$ (which leads to the conclusion)  
by the following commutative diagram: 
\[\xymatrix{
&\Ker \dbar \text{ in } L^{n,q}_{(2)}(F)_{h_{\e_{0}},\omega}  
\ar[r]^{q_{2}\ \ }&  
\dfrac{\Ker \dbar}{\Image \dbar} 
\text{ of }  L^{n,q}_{(2)}(F)_{h_{\e_{0}},\omega} 
\ar[r]^{\cong\ \ \ }_{f_{2}\ \ }
&\check{H}^{q}(X, K_X\otimes F \otimes \I{h_{\e_{0}}})\\
&\Ker \dbar \text{ in } L^{n,q}_{(2)}(F)_{h, \omega}
\ar[u]_{} \ar[r]^{q_{3}\ \ } 
&\dfrac{\Ker \dbar}{\Image \dbar} 
\text{ of }  L^{n,q}_{(2)}(F)_{h, \omega}  
\ar[r]^{\cong\ \ \ }_{f_{3}\ \ }  
& \check{H}^{q}(X, K_X\otimes F \otimes \I{h}). 
\ar@{=}[u]
}\]
\end{proof}

At the end of this step, 
we prove Proposition \ref{f-prop5.7}. 

\begin{prop}\label{f-prop5.7}
If we have 
$$
\varliminf_{\e \to 0} \varliminf_{\delta \to 0} \|su_{\e, \delta}\|_{h_{\e}h_M, \omega_{\e,\delta}} =0, 
$$
then the weak limit $\alpha$ is zero. 
In particular, 
the cohomology class $\eta$ is zero by Proposition \ref{f-prop5.6}. 
\end{prop}

\begin{proof}[Proof of Proposition \ref{f-prop5.7}.]
In the proof, we compare the 
norm of $u_{\e, \delta}$ with the norm of $su_{\e, \delta}$. 
For this purpose, we define $Y^{k}_{\e_{0}}$ to be  
\begin{align*}
Y^{k}_{\e_{0}}:=\{y\in Y_{\e_{0}} \, 
\mid\, |s|_{h_M}>1/k \text{ at } y\} 
\end{align*}
for $k \gg 0$. 
Note the subset $Y^{k}_{\e_{0}}$ is an open set in $Y_{\e_{0}}$. 
It follows that the restriction 
$\alpha_{\e}|_{Y^{k}_{\e_{0}}}$ also weakly converges 
to $\alpha|_{Y^{k}_{\e_{0}}}$ 
in $L^{n,q}_{(2)}( Y^{k}_{\e_{0}},F)_{h_{\e_{0}}, \omega}$ 
since the restriction map $L^{n,q}_{(2)}(F)_{h_{\e_{0}}, \omega} 
\to L^{n,q}_{(2)}( Y^{k}_{\e_{0}},F)_{h_{\e_{0}}, \omega}$ is 
a bounded operator and  $\alpha_{\e}$ weakly converges 
to $\alpha$ 
in $L^{n,q}_{(2)}(F)_{h_{\e_{0}}, \omega}$. 
Since the norm is lower semicontinuous with respect to 
the weak convergence, 
we obtain the estimate for the $L^2$-norm on $Y^{k}_{\e_{0}}$
\begin{align*}
\| \alpha \|_{Y^{k}_{\e_{0}}, h_{\e_{0}}, \omega}
\leq \varliminf_{\e \to 0} 
\| \alpha_{\e} \|_{ Y^{k}_{\e_{0}}, h_{\e_{0}}, \omega}
\leq \varliminf_{\e \to 0} 
\| \alpha_{\e} \|_{ Y^{k}_{\e_{0}}, h_{\e}, \omega}
\end{align*}
by property (b). 
By the same argument, 
the restriction $u_{\e,\delta}|_{Y^{k}_{\e_{0}}}$ weakly converges to 
$\alpha_{\e}|_{Y^{k}_{\e_{0}}}$ 
in $L^{n,q}_{(2)}(Y^{k}_{\e_{0}},F)_{\e, \delta_{0}}$, 
and thus we obtain 
\begin{align*}
\| \alpha_{\e} \|_{Y^{k}_{\e_{0}},\e, \delta_{0}}
\leq \varliminf_{\delta \to 0} 
\| u_{\e,\delta} \|_{Y^{k}_{\e_{0}}, \e, \delta_{0}}
\leq \varliminf_{\delta \to 0} 
\| u_{\e,\delta} \|_{Y^{k}_{\e_{0}}, \e, \delta}
\end{align*}
by Lemma \ref{f-lem2.4}. 
As $\delta_{0} \to 0$ in the above inequality, 
we have 
$$
\| \alpha_{\e} \|_{Y^{k}_{\e_{0}},h_{\e}, \omega}
\leq 
\varliminf_{\delta_{0} \to 0}
\| \alpha_{\e} \|_{Y^{k}_{\e_{0}},\e, \delta_{0}}
\leq \varliminf_{\delta \to 0} 
\| u_{\e,\delta} \|_{Y^{k}_{\e_{0}}, \e, \delta} 
$$ 
by Fatou's lemma (see the argument in Proposition \ref{f-prop5.4}).  
These inequalities yield     
\begin{align*}
\| \alpha \|_{Y^{k}_{\e_{0}}, h_{\e_{0}}, \omega}
\leq \varliminf_{\e \to 0} \varliminf_{\delta \to 0}
\| u_{\e,\delta} \|_{Y^{k}_{\e_{0}}, 
\e, \delta}. 
\end{align*}
On the other hand, 
it follows that 
\begin{align*}
\| u_{\e,\delta} \|_{Y^{k}_{\e_{0}}, 
\e, \delta}
\leq k \| su_{\e,\delta} \|_{Y^{k}_{\e_{0}}, 
h_{\e}h_M, \omega_{\e,\delta}}
\leq k \| su_{\e,\delta} \|_{h_{\e}h_M, \omega_{\e,\delta}} 
\end{align*}
since the inequality $1/k < |s|_{h_M}$ holds on $Y^{k}_{\e_{0}}$. 
This implies that 
$\alpha=0$ on $Y^{k}_{\e_{0}}$ for an arbitrary $k\gg 0$. 
From $\underset{k\gg 0}{\bigcup}Y^{k}_{\e_{0}}=Y_{\e_{0}}\setminus \{s=0\}$, 
we obtain the desired conclusion. 
\end{proof}
\end{step}

\begin{step}\label{f-st3}
The purpose of this step is to prove the following proposition: 
\begin{prop}\label{f-prop5.8}
$$\lim_{\e \to 0} \varlimsup_{\delta \to 0}
\| \dbar^{*}_{\e, \delta} s u_{\e, \delta} 
\|_{h_{\e}h_M, \omega_{\e,\delta}}=0.
$$

\end{prop}
\begin{proof}[Proof of Proposition \ref{f-prop5.8}]
In the proof, we will often use \eqref{eq5.3}. 
By applying Bochner--Kodaira--Nakano's identity and the density lemma to 
$u_{\e, \delta}$ and $s u_{\e, \delta}$ 
(see \cite[Proposition 2.8]{matsumura1}), 
we obtain 
\begin{align}
0 &= \lla \sqrt{-1}\Theta_{h_{\e}}(F)
\Lambda_{\omega_{\e, \delta}} u_{\e,\delta}, u_{\e,\delta}
  \rra_{\e,\delta} + \|D'^{*}_{\e,\delta}u_{\e,\delta} \|^{2}_{\e,\delta}, \label{eq5.4}\\
\| \dbar^{*}_{\e, \delta} s u_{\e, \delta} 
\|^2_{h_{\e}h_M, \omega_{\e,\delta}} &= \lla \sqrt{-1}\Theta_{h_{\e}h_{M}}(F\otimes M)
\Lambda_{\omega_{\e, \delta}} su_{\e, \delta}, su_{\e, \delta}
  \rra_{h_{\e}h_M, \omega_{\e,\delta}} + 
  \|D'^{*}_{\e, \delta}su_{\e, \delta} \|^{2}_{h_{\e}h_M, \omega_{\e,\delta}},    \label{eq5.5}
\end{align}
where $D'^{*}_{\e, \delta}$ 
is the adjoint operator of the $(1,0)$-part 
of the Chern connection $D_{h_{\e}}$. 
Here we used the fact that $u_{\e, \delta}$ is harmonic 
and $\dbar (s u_{\e,\delta})=s\dbar u_{\e,\delta}=0$. 
Now we have 
$$
\sqrt{-1}\Theta_{h_{\e}}(F) \geq b \sqrt{-1}\Theta_{h_M}(M)  -\e \omega 
\geq -\e \omega \geq -\e \omega_{\e,\delta}  
$$ 
by property (d) and property (B). 
Hence, the integrand $g_{\e, \delta}$ of the first term 
of \eqref{eq5.4} satisfies 
\begin{equation}\label{eq5.6}
 -\e q |u_{\e, \delta}|^{2}_{\e, \delta} 
 \leq g_{\e, \delta}:= \langle \sqrt{-1}\Theta_{h_{\e}}(F)
\Lambda_{\omega_{\e, \delta}} u_{\e,\delta}, u_{\e,\delta}
  \rangle_{\e,\delta}. 
\end{equation}
For the precise 
argument, see \cite[Step 2 in the proof of Theorem 3.1]{matsumura4}. 
Then by \eqref{eq5.4}, we can easily see   
\begin{align*}
\lim_{\e \to 0} \varlimsup_{\delta \to 0} 
\Big( \int_{\{g_{\e, \delta} \geq 0\}} g_{\e, \delta}\,dV_{\omega_{\e, \delta}}
+\|D'^{*}_{\e,\delta}u_{\e,\delta} \|^{2}_{\e,\delta} \Big)
&= \lim_{\e \to 0} \varlimsup_{\delta \to 0} 
\Big(-\int_{\{g_{\e, \delta} \leq 0\}} g_{\e, \delta}\, dV_{\omega_{\e, \delta}}\Big)\\
&\leq \lim_{\e \to 0} \varlimsup_{\delta \to 0} 
\Big(\e q \int_{\{g_{\e, \delta} \leq 0\}}  |
u_{\e, \delta}|^{2}_{\e, \delta}\, dV_{\omega_{\e, \delta}}\Big)\\
&\leq  \lim_{\e \to 0} \varlimsup_{\delta \to 0} 
\Big(\e q \|u_{\e, \delta} \|^2_{\e, \delta}\Big)=0.  
\end{align*}
Here we used \eqref{eq5.3} in the last equality. 

On the other hand, 
by $\sqrt{-1}\Theta_{h_{\e}}(F) 
\geq b \sqrt{-1}\Theta_{h_M}(M) -\e \omega_{\e, \delta}$, 
we have 
\begin{align*}
&\lla \sqrt{-1}\Theta_{h_{\e}h_{M}}(F\otimes M)
\Lambda_{\omega_{\e, \delta}} 
su_{\e, \delta}, su_{\e, \delta} \rra_{h_{\e}h_M, \omega_{\e,\delta}}\\ 
\leq& 
\big(1+\frac{1}{b}\big)
\int_{Y_{\e}} |s|^2_{h_M}g_{\e,\delta}\,dV_{\omega_{\e, \delta}}
+\frac{\e q}{b} \int_{Y_{\e}} 
|s|^2_{h_M}|u_{\e, \delta}|^{2}_{\e, \delta} \,dV_{\omega_{\e, \delta}}\\
\leq& \big(1+\frac{1}{b}\big)\sup_{X} 
|s|^2_{h_M} \Big\{
\int_{\{g_{\e, \delta} \geq 0\}} g_{\e,\delta}\,dV_{\omega_{\e, \delta}}+
\frac{\e q}{b} \sup_{X} |s|^2_{h_M} \|u_{\e,\delta}\|^2_{\e,\delta} \Big\}. 
\end{align*}
Furthermore, since $D'^{*}_{\e, \delta}$ can be expressed as 
$D'^{*}_{\e, \delta}=-*\dbar* $ by the Hodge star operator $*$ 
with respect to $\omega_{\e,\delta}$,  
we have 
\begin{align*}
\|D'^{*}_{\e, \delta} su_{\e, \delta} \|^{2}_{h_{\e}h_M, \omega_{\e,\delta}}
=\|s D'^{*}_{\e, \delta} u_{\e, \delta} \|^{2}_{h_{\e}h_M, \omega_{\e,\delta}}
\leq \sup_{X} |s|^2_{h_M} \|D'^{*}_{\e, \delta} u_{\e, \delta} \|^{2}_{\e, \delta}. 
\end{align*}
The right-hand side of \eqref{eq5.5} can be shown to converge to zero 
by the first half argument and these inequalities. 
\end{proof}
\end{step}

\begin{step}\label{f-st4}
In this step, we construct solutions $v_{\e,\delta}$ of 
the $\dbar$-equation $\dbar v_{\e,\delta}=s u_{\e,\delta}$ with suitable $L^2$-norm, 
and we finish the proof of Theorem \ref{f-thm5.1}. 
The proof of the following proposition is 
a slight variant of that of \cite[Theorem 5.9]{matsumura4}.  

\begin{prop}\label{f-prop5.9}
There exist  $F$-valued $(n, q-1)$-forms  $w_{\e, \delta}$ 
on $Y_{\e}$ with the following properties:  
\begin{itemize}
\item[$\bullet$] $\dbar w_{\e,\delta}=u-u_{\e, \delta}$. 
\item[$\bullet$] $\varlimsup_{\delta \to 0} \| w_{\e,\delta}\|_{\e, \delta} $ 
can be bounded by a constant independent of $\e$. 
\end{itemize}
\end{prop}

Before we begin to prove Proposition \ref{f-prop5.9}, 
we recall the content in \cite[Section 5]{matsumura4} with our notation. 
For a finite open cover $\mathcal{U}:=\{B_{i}\}_{i \in I}$ of $X$ 
by sufficiently small Stein open sets $B_{i}$,  
we can construct 
\begin{align*}
f_{\e,\delta}\colon \Ker \dbar \text{ in } L^{n,q}_{(2)}(F)_{\e,\delta} 
\xrightarrow{\quad  \quad }
\Ker \mu \text{ in } 
C^{q}(\mathcal{U}, K_X \otimes F \otimes \I{h_{\e}}) 
\end{align*}
such that $f_{\e,\delta}$ induces the de Rham--Weil isomorphism 
\begin{align}\label{eq5.7}
\overline{f_{\e,\delta}} \colon \dfrac{\Ker \dbar}{\Image \dbar} 
\text{ of } L^{n,q}_{(2)}(F)_{\e,\delta}  
\xrightarrow{\quad \cong \quad }
\dfrac{\Ker \mu}{\Image \mu} 
\text{ of } C^{q}(\mathcal{U}, K_X\otimes F \otimes \I{h_{\e}}).  
\end{align}
Here $C^{q}(\mathcal{U}, K_X \otimes F \otimes \I{h_{\e}})$ is 
the space of $q$-cochains calculated by $\mathcal{U}$ and 
$\mu$ is the coboundary operator. 
We remark that $C^{q}(\mathcal{U}, K_X \otimes F \otimes \I{h_{\e}})$ 
is a Fr\'echet space with respect to the seminorm $p_{K_{i_{0}...i_{q}}}(\bullet)$ 
defined to be 
\begin{equation*}
p_{K_{i_{0}...i_{q}}}(\{\beta_{i_{0}...i_{q}}\})^{2}:=
\int_{K_{i_{0}...i_{q}}} |\beta_{i_{0}...i_{q}}|_{h_{\e}, \omega}^{2}\, dV_{\omega}
\end{equation*}
for a relatively compact set 
$K_{i_{0}...i_{q}} \Subset B_{i_{0}...i_{q}}:=B_{i_{0}}\cap \dots \cap B_{i_{q}}$ 
(see \cite[Theorem 5.3]{matsumura4}). 
The construction of $f_{\e, \delta}$ is essentially the same as 
in the proof of \cite[Proposition 5.5]{matsumura4}. 
The only difference is 
that we use Lemma \ref{f-lem5.12} instead of \cite[Lemma 5.4]{matsumura4} 
when we locally solve the $\dbar$-equation to construct $f_{\e, \delta}$.
Lemma \ref{f-lem5.12} will be given at the end of this step. 
We prove Proposition \ref{f-prop5.9}  
by replacing some constants appearing in the proof of 
\cite[Theorem 5.9]{matsumura4} 
with $C_{\e,\delta}$ appearing in Lemma \ref{f-lem5.12}. 

\begin{proof}[Proof of Proposition \ref{f-prop5.9}]
We put $U_{\e,\delta}:=u-u_{\e,\delta} \in \Image 
\dbar \subset L^{n,q}_{(2)}(F)_{\e,\delta}$. 
Then there exist the $F$-valued $(n,q-k-1)$-forms 
$\beta^{\e, \delta}_{i_{0}\dots i_{k}}$ 
on $B_{i_{0}\dots i_{k}} \setminus Z_{\e}$ 
satisfying 
\[
  (*) \left\{ \quad
  \begin{array}{ll}
\vspace{0.2cm}
\dbar \beta^{\e, \delta}_{i_{0}} &=U_{{\e, \delta}} |_{B_{i_{0}}\setminus Z_{\e}},  \\
\dbar \{ \beta^{\e, \delta}_{i_{0}i_{1}} \}&=\mu  \{\beta^{\e, \delta}_{i_{0}}\},  \\
\dbar \{ \beta^{\e, \delta}_{i_{0}i_{1}i_{2}} \}&
=\mu   \{\beta^{\e, \delta}_{i_{0}i_{1}}\},  \\
 & \vdots   \\
\dbar \{ \beta^{\e, \delta}_{i_{0}\dots i_{q-1}} \}
&=\mu   \{\beta^{\e, \delta}_{i_{0}\dots i_{q-2}}\}, \\
f_{\e, \delta}(U_{\e, \delta})&=\mu 
\{\beta^{\e, \delta}_{i_{0}\dots i_{q-1}}\}. 
  \end{array} \right.
\]
Here $\beta^{\e, \delta}_{i_{0}\dots i_{k}}$ 
is the solution of the above equation whose 
norm is minimum among all the solutions 
(see the construction of $f_{\e,\delta}$ 
in \cite[Proposition 5.5]{matsumura4}). 
For example, $\beta^{\e, \delta}_{i_{0}}$ is the solution 
of $\dbar \beta^{\e, \delta}_{i_{0}}=U_{{\e, \delta}}$ on  
${B_{i_{0}}\setminus Z_{\e}}$ whose norm 
$\|\beta^{\e, \delta}_{i_{0}} \|_{\e,\delta}$ 
is minimum among all the solutions. 
In particular  
$\|\beta^{\e, \delta}_{i_{0}}\|^{2}_{\e,\delta} \leq 
C_{\e, \delta} \|U_{\e, \delta} \|^2_{B_{i_{0}}, \e,\delta}
\leq C_{\e, \delta} \|U_{\e, \delta} \|^2_{\e,\delta}$ 
holds for some constant $C_{\e,\delta}$ by Lemma \ref{f-lem5.12}, 
where $C_{\e,\delta}$ is a constant such that 
$\varlimsup_{\delta \to 0} C_{\e,\delta}$ (is finite and) is independent of $\e$. 
Similarly, $\beta^{\e, \delta}_{i_{0}i_{1}}$ is 
the solution of $\dbar \beta^{\e, \delta}_{i_{0}i_{1}}= 
(\beta^{\e, \delta}_{i_{1}}- \beta^{\e, \delta}_{i_{0}})$ on 
$B_{i_{0}i_{1}}\setminus Z_{\e}$ 
and the norm 
$$
\|\beta^{\e, \delta}_{i_{0}i_{1}}\|^2_{\e,\delta}:=
\int_{B_{i_{0}i_{1}}\setminus Z_{\e}} 
|\beta^{\e, \delta}_{i_{0}i_{1}}|^{2}_{\e,\delta}\, dV_{\e,\delta}
$$
is minimum among all the solutions. 
In particular, 
$\|\beta^{\e, \delta}_{i_{0}i_{1}}\|^{2}_{\e,\delta} \leq 
D_{\e, \delta} \|(\beta^{\e, \delta}_{i_{1}}- \beta^{\e, \delta}_{i_{0}})\|^2_{\e,\delta}$ holds 
for some constant $D_{\e,\delta}$ by Lemma \ref{f-lem5.12}. 
Of course $D_{\e,\delta}$ is a constant such that 
$\varlimsup_{\delta \to 0} D_{\e,\delta}$ (is finite and) is  
independent of $\e$.
Hence we have 
$$
\|\beta^{\e, \delta}_{i_{0}i_{1}}\|_{\e,\delta} \leq 
D_{\e, \delta}^{1/2} \|(\beta^{\e, \delta}_{i_{1}}- 
\beta^{\e, \delta}_{i_{0}})\|_{\e,\delta}
\leq 
2C^{1/2}_{\e, \delta}D^{1/2}_{\e,\delta} \|U_{\e, \delta} \|_{\e,\delta}
\leq 4C^{1/2}_{\e, \delta}D^{1/2}_{\e,\delta} \|u \|_{h, \omega}
$$
by \eqref{eq5.3}. 
From now on, the notation $C_{\e,\delta}$ denotes 
a (possibly different) constant such that  
$\varlimsup_{\delta \to 0} C_{\e,\delta}$ can be bounded by 
a constant independent of $\e$.
By repeating this process, 
we have 
\begin{align*}
\|\beta^{\e, \delta}_{i_{0}\dots i_{k}}\|^2_{\e, \delta}\leq 
C_{\e, \delta} \|u\|^2_{h,\omega}. 
\end{align*}
Moreover, by property (c), we have 
$$
\alpha_{\e, \delta}:=f_{\e, \delta}(U_{\e, \delta})=
\mu \{\beta^{\e, \delta}_{i_{0}\dots i_{q-1}}\} 
\in C^{q}(\mathcal{U}, K_X\otimes F\otimes \I{h_{\e}})
=C^{q}(\mathcal{U}, K_X\otimes F\otimes \I{h}).
$$
\begin{claim}\label{f-claim}
There exist subsequences $\{\e_{k}\}_{k=1}^{\infty}$ and 
$\{\delta_{\ell}\}_{\ell=1}^{\infty}$ with the following properties: 
\begin{itemize}
\item[$\bullet$] $\alpha_{\e_{k},\delta_{\ell}} \to \alpha_{\e_{k},0}$ 
in $C^{q}(\mathcal{U}, K_X\otimes F \otimes \I{h})$
as 
$\delta_{\ell} \to 0$.
\item[$\bullet$] $\alpha_{\e_{k},0} \to 
\alpha_{0,0}$ in $C^{q}(\mathcal{U}, 
K_X\otimes F \otimes \I{h})$ as $\e_{k} \to 0$. 
\end{itemize}
Moreover, the limit $\alpha_{0,0}$ belongs to 
$B^{q}(\mathcal{U}, K_X\otimes F \otimes \I{h}):= \Image\mu$. 
\end{claim}
\begin{proof}[Proof of Claim]
By construction, the norm $\|a_{\e,\delta}\|_{B_{i_{0}\dots i_{q}}, \e, \delta}$ 
of a component $a_{{\e,\delta}}:=\alpha^{\e,\delta}_{{i_{0}\dots i_{q}}}$ 
of $\alpha_{\e,\delta}=\{ \alpha^{\e,\delta}_{{i_{0}\dots i_{q}}} \} $ 
can be bounded by a constant $C_{\e, \delta}$.
Note that $a_{\e,\delta}$ can be regarded as a holomorphic function 
on $B_{i_{0}\dots i_{q}} \setminus Z_{\e}$
with bounded $L^2$-norm 
since it is a $\dbar$-closed $F$-valued $(n,0)$-form such that 
$\|a_{\e,\delta}\|_{B_{i_{0}\dots i_{q}}, \e, \delta}<\infty$ 
(see Lemma \ref{f-lem2.4}). 
Hence $a_{\e,\delta}$ can be extended from 
$B_{i_{0}\dots i_{q}}\setminus Z_{\e}$ 
to $B_{i_{0}\dots i_{q}}$ by the Riemann extension theorem.
The sup-norm $\sup_{K}|a_{\e,\delta}|$ is 
uniformly bounded with respect to $\delta$
for every $K \Subset B_{i_{0}\dots i_{q}}$ since 
the local sup-norm of holomorphic functions can be bounded by the $L^{2}$-norm. 
By Montel's theorem, we can take 
a subsequence $\{\delta_{\ell}\}_{\ell=1}^{\infty}$ 
with the first property. 
This subsequence may depend on $\e$, 
but we can take $\{\delta_{\ell}\}_{\ell=1}^{\infty}$ 
independent of (countably many) $\e$. 
Then the norm of the limit $a_{\e,0}$ is uniformly bounded with respect to $\e$ 
since $\varlimsup_{\delta \to 0}C_{\e,\delta}$ can be 
bounded by a constant independent of $\e$ (see Lemma \ref{f-lem5.12}). 
Therefore, by applying Montel's theorem again, 
we can take a subsequence $\{\e_{k}\}_{k=1}^{\infty}$ 
with the second property. 
We remark that the convergence with respect to the sup-norm  
implies the convergence with respect to the local $L^2$-norm $p_{K}(\bullet)$  
(see \cite[Lemma 5.2]{matsumura4}). 

It is easy to check the latter conclusion. 
Indeed, it follows that 
$\alpha_{\e, \delta} = f_{\e, \delta}(U_{\e,\delta})\in \Image\mu$ since  
$U_{\e,\delta} \in \Image \dbar \subset L^{n,q}_{(2)}(F)_{\e,\delta}$ and 
$f_{\e,\delta}$ induces the de Rham--Weil isomorphism. 
By \cite[Lemma 5.7]{matsumura4}, the subspace $\Image\mu$ 
is closed. 
Therefore, we obtain the latter conclusion. 
\end{proof}

Now, we construct solutions $\gamma_{\e,\delta}$ of 
the equation $\mu \gamma_{\e,\delta} = \alpha_{\e,\delta}$ 
with suitable $L^2$-norm. 
For simplicity, we continue to use the same notation for the subsequences  
in Claim. 
By the latter conclusion of the claim, 
there exists  
$\gamma \in C^{q-1}(\mathcal{U}, K_X\otimes F \otimes \I{h})$ 
such that $\mu \gamma = \alpha_{0,0}$. 
The coboundary operator 
\begin{equation*}
\mu\colon C^{q-1}(\mathcal{U}, K_X\otimes F\otimes \I{h}) \to 
B^{q}(\mathcal{U}, K_X\otimes F\otimes \I{h})= \Image\mu
\end{equation*}
is a surjective bounded operator between Fr\'echet spaces 
(see \cite[Lemma 5.7]{matsumura4}), and thus  
it is an open map by the open mapping theorem. 
Therefore $\mu (\Delta_{K})$ is an open neighborhood of 
the limit $\alpha_{0,0}$ in $\Image\mu$, 
where $\Delta_{K}$ is the open bounded neighborhood of $\gamma$ in 
$C^{q-1}(\mathcal{U}, K_X\otimes F \otimes \I{h})$ 
defined to be  
\begin{equation*}
\Delta_{K}:=\{ \beta \in C^{q-1}(\mathcal{U}, 
K_X\otimes F \otimes \I{h}) 
\mid \ p_{K_{i_{0}...i_{q-1}}}(\beta-\gamma) < 1 \} 
\end{equation*}
for a family $K:=\{K_{i_{0}...i_{q-1}}\}$ 
of relatively compact sets 
$K_{i_{0}...i_{q-1}} \Subset B_{i_{0}...i_{q-1}}$. 
We have $\alpha_{\e, \delta} \in \mu (\Delta_{K})$ for sufficiently small $\e, \delta>0$  
since $\alpha_{\e, \delta}$ converges to $\alpha_{0,0}$. 
Since $\Delta_{K}$ is bounded, 
we can obtain $\gamma_{\e, \delta} \in 
C^{q-1}(\mathcal{U}, K_X\otimes F \otimes \I{h})$ 
such that 
\begin{align*}
\mu \gamma_{\e, \delta}= \alpha_{\e, \delta} \text{\quad and \quad}
p_{K_{i_{0}...i_{q-1}}}(\gamma_{\e, \delta})^{2}\leq C_{K} 
\end{align*}
for some positive constant $C_{K}$. 
The above constant $C_{K}$ 
depends on the choice of $K$, $\gamma$, but does not depend on $\e, \delta$. 

By the same argument as in \cite[Claim 5.11 and Claim 5.13]{matsumura4}, 
we can obtain $F$-valued $(n,q-1)$-forms $w_{\e, \delta}$ with the desired properties. 
The strategy is as follows:
The inverse map $\overline{g_{\e,\delta}}$ of $\overline{f_{\e,\delta}}$ 
is explicitly constructed by using a partition of unity 
(see the proof of \cite[Proposition 5.5]{matsumura4} and 
\cite[Remark 5.6]{matsumura4}). 
We can easily see that 
$g_{\e, \delta}(\mu \gamma_{\e, \delta})=\dbar v_{\e, \delta}$ and 
$g_{\e, \delta}(\alpha_{\e, \delta})=U_{\e, \delta}+\dbar \widetilde{v}_{\e, \delta}$ 
hold for some $v_{\e, \delta}$ and $\widetilde{v}_{\e, \delta}$ 
by the de Rham--Weil isomorphism. 
In particular, we have $U_{\e, \delta}
=\dbar(v_{\e, \delta} - \widetilde{v}_{\e, \delta})$
by $\mu \gamma_{\e, \delta}=\alpha_{\e, \delta}$.  
The important point here is that 
we can explicitly compute $v_{\e, \delta}$ and $\widetilde{v}_{\e, \delta}$ 
by using the partition of unity, $\beta^{\e,\delta}_{i_{0}...i_{k}}$, 
and $\gamma_{\e,\delta}$. 
From this explicit expression, we obtain the $L^2$-estimate for 
$v_{\e, \delta}$ and $\widetilde{v}_{\e, \delta}$.
See \cite[Claim 5.11 and 5.13]{matsumura4} for the precise argument. 
\end{proof}

\begin{prop}\label{f-prop5.10}
There exist $F\otimes M$-valued $(n, q-1)$-forms  
$v_{\e,\delta}$ 
on $Y_{\e}$  
with the following properties:  
\begin{itemize}
\item[$\bullet$] $\dbar v_{\e,\delta}=su_{\e, \delta}$. 
\item[$\bullet$] $\varlimsup_{\delta \to 0} \| v_{\e,\delta}\|_{h_{\e}h_M, \omega_{\e,\delta}} $ 
can be bounded by a constant independent of $\e$. 
\end{itemize}
\end{prop}
\begin{proof}[Proof of Proposition \ref{f-prop5.10}]
Since the cohomology class of $su$ is assumed to be zero 
in $H^{q}(X, K_X\otimes F \otimes \I{h}\otimes M)$,  
there exists an $F\otimes M$-valued $(n, q-1)$-form $v$ 
such that $\dbar v =  su$ and 
$\|v \|_{h, \omega} < \infty$.  
For $w_{\e,\delta}$ satisfying the properties in Proposition \ref{f-prop5.9},  
by putting $v_{\e,\delta}:= -s w_{\e,\delta} + v$,  
we have $\dbar v_{\e,\delta} = su_{\e,\delta}$. 
Furthermore, an easy computation yields  
\begin{align*}
\|v_{\e, \delta}\|_{h_{\e}h_M, \omega_{\e,\delta}} \leq 
\|s w_{\e, \delta}\|_{h_{\e}h_M, \omega_{\e,\delta}} + 
\|v \|_{h_{\e}h_M, \omega_{\e,\delta}}
\leq \sup_{X}|s|_{h_M} \|w_{\e, \delta} \|_{\e, \delta} + 
\|v \|_{h_{\e}h_M, \omega_{\e,\delta}}. 
\end{align*}
By Lemma \ref{f-lem2.4}, property (b), and property (B), 
we have $\| v \|_{h_{\e}h_M, \omega_{\e,\delta}}
\leq \|v \|_{h, \omega} < \infty$. 
This completes the proof.  
\end{proof}

The following proposition completes 
the proof of Theorem \ref{f-thm5.1} 
(see Proposition \ref{f-prop5.7}).
 
\begin{prop}\label{f-prop5.11}
\begin{align*}
\lim_{\e \to 0} \varlimsup_{\delta \to 0} \|s 
u_{\e, \delta}\|_{h_{\e}h_M, \omega_{\e,\delta}}=0. 
\end{align*}
\end{prop}
\begin{proof}[Proof of Proposition \ref{f-prop5.11}]
For the solution 
$v_{\e,\delta}$ satisfying the properties in Proposition \ref{f-prop5.10}, 
it is easy to see 
\begin{align*}
\lim_{\e \to 0} \varlimsup_{\delta \to 0} \|s 
u_{\e, \delta}\|^2_{h_{\e}h_M, \omega_{\e,\delta}}
&=\lim_{\e \to 0} \varlimsup_{\delta \to 0} 
\lla \dbar^{*}_{\e,\delta} s 
u_{\e, \delta}, v_{\e,\delta}\rra_{h_{\e}h_M, \omega_{\e,\delta}}\\
&\leq \lim_{\e \to 0} \varlimsup_{\delta \to 0}
\| \dbar^{*}_{\e,\delta} s u_{\e, \delta}\|_{h_{\e}h_M, 
\omega_{\e,\delta}} \|v_{\e,\delta}\|_{h_{\e}h_M, \omega_{\e,\delta}}. 
\end{align*}
Proposition \ref{f-prop5.8} and Proposition \ref{f-prop5.10} assert that 
the right-hand side is zero. 
\end{proof}

We close this step with the following lemma: 

\begin{lem}[{cf.~\cite[4.1\,Th\'eor\`eme]{demailly-dbar}}]\label{f-lem5.12}
Assume that $B$ is a Stein open set in $X$ such that 
$\omega_{\e, \delta}=\deldel 
(\Psi + \delta \Psi_{\e})$ on a neighborhood of $\overline B$. 
Then for an arbitrary 
$\alpha \in \Ker \dbar \subset L^{n, q}_{(2)}(B\setminus Z_{\e}, F)_{\e, \delta}$, 
there exist $\beta \in L^{n, q-1}_{(2)}(B\setminus Z_{\e}, F)_{\e, \delta}$ 
and a positive constant $C_{\e, \delta}$ $($independent of $\alpha$$)$ 
such that 
\begin{align*}
&\bullet \dbar \beta = \alpha \text{\quad and \quad}
\|\beta \|^{2}_{\e, \delta}\leq C_{\e, \delta}  \|\alpha \|^{2}_{\e, \delta},\\ 
&\bullet \varlimsup_{\delta \to 0} 
C_{\e,\delta} \text{ $($is finite and$)$  is independent of }\e. 
\end{align*}
\end{lem}
\begin{proof}[Proof of Lemma \ref{f-lem5.12}]
We may assume $\e < 1/2$ since $0<\e \ll 1$. 
For the singular Hermitian metric $H_{\e,\delta}$ on $F$ defined by 
$H_{\e,\delta}:=h_{\e} e^{-(\Psi + \delta \Psi_{\e})}$, 
the curvature satisfies 
\begin{align*}
\sqrt{-1}\Theta_{H_{\e,\delta}}(F)&= \sqrt{-1}\Theta_{h_{\e}}(F) + 
\deldel (\Psi + \delta \Psi_{\e})  \geq -\e \omega + \omega_{\e,\delta}
\geq (1-\e) \omega_{\e,\delta}
\geq  \frac{1}{2} \omega_{\e,\delta}
\end{align*}
by property (B) and 
$\sqrt{-1}\Theta_{h_{\e}}(F) \geq -\e \omega$.  
The $L^{2}$-norm $\|\alpha \|_{H_{\e,\delta}, 
\omega_{\e,\delta}}$ with respect to $H_{\e,\delta}$ 
and $\omega_{\e,\delta}$ is finite 
since the function $\Psi + \delta \Psi_{\e} $ is bounded and 
$\|\alpha \|_{\e,\delta}$ is finite. 
Therefore, from the standard $L^{2}$-method for the $\dbar$-equation 
(for example see \cite[4.1\,Th\'eor\`eme]{demailly-dbar}), 
we obtain a solution $\beta$ of the $\dbar$-equation 
$\dbar \beta =\alpha$ with  
\begin{equation*}
\|\beta \|^{2}_{H_{\e,\delta},\omega_{\e,\delta}} 
\leq \frac{2}{q} \|\alpha \|^{2}_{H_{\e,\delta},\omega_{\e,\delta}}.  
\end{equation*}
Then we can easily see that 
\begin{equation*}
\|\beta \|^2_{\e, \delta}  \leq 
\frac{2}{q}
\frac{ \sup_{B} e^{-(\Psi + \delta \Psi_{\e})} }
{\inf_{B} e^{-(\Psi + \delta \Psi_{\e})}}
\|\alpha \|^2_{\e, \delta}. 
\end{equation*}
This completes the proof by property (B). 
\end{proof}

\begin{rem}\label{f-rem5.13}
In Lemma \ref{f-lem5.12}, we take a solution 
$\beta_0\in L^{n, q-1}_{(2)}(B\setminus Z_{\e}, F)_{\e, \delta}$ 
of the equation $\dbar \beta=\alpha$. 
Then $\beta_0$ is uniquely decomposed as follows: 
$$
\beta_0=\beta_1+\beta_2
\quad \text{for }\beta_1\in \Ker \dbar
\text{ and } 
\beta_2\in (\Ker\dbar )^{\perp}. 
$$
We can easily check that $\beta_2$ is a unique solution of $\dbar \beta=
\alpha$ whose norm is the minimum among all the solutions. 
\end{rem}
\end{step}
Thus we finish the proof of Theorem \ref{f-thm5.1}. 
\end{proof}

\section{Twists by Nakano semipositive vector bundles}\label{f-sec6}

We have already known that 
some results for $K_X$ can be generalized 
for $K_X\otimes E$, 
where $E$ is a Nakano semipositive vector bundle 
on $X$ (see, for example, \cite{takegoshi}, \cite{mour}, 
and \cite{fujisawa}). Let us recall the 
definition of Nakano semipositive vector bundles. 

\begin{defn}[Nakano semipositive vector bundles]\label{f-def6.1}
Let $E$ be a holomorphic vector bundle on a complex manifold $X$. 
If $E$ admits a smooth Hermitian metric $h_E$ 
such that the curvature form $\sqrt{-1}\Theta_{h_E}(E)$ defines 
a positive semi-definite Hermitian form on each fiber of the vector bundle $E\otimes T_X$, 
where $T_X$ is the holomorphic tangent bundle of $X$, 
then $E$ is called a Nakano semipositive vector bundle. 
\end{defn}

\begin{ex}[Unitary flat vector bundles]\label{f-ex6.2} 
Let $E$ be a holomorphic 
vector bundle on a complex manifold $X$.  
If $E$ admits a smooth Hermitian metric $h_E$ such that $(E, h_E)$ is flat, 
that is, $\sqrt{-1}\Theta_{h_E}(E)=0$, then 
$E$ is Nakano semipositive. 
\end{ex}

For the proof of Theorem \ref{f-thm1.12}, 
we need the following lemmas on Nakano semipositive vector bundles. 
However, these lemmas easily follow from the 
definition of Nakano semipositive vector bundles, 
and thus, we omit the proof. 

\begin{lem}\label{f-lem6.3}
Let $E$ be a Nakano semipositive vector bundle on a 
complex manifold $X$. 
Let $H$ be a smooth divisor on $X$. 
Then $E|_H$ is a Nakano semipositive vector bundle on $H$. 
\end{lem}

\begin{lem}\label{f-lem6.4}
Let $q\colon Z\to X$ be an \'etale morphism between complex manifolds. 
Let $(E, h_E)$ be a Nakano semipositive vector bundle on $X$. 
Then $(q^*E, q^*h_E)$ is a Nakano semipositive vector bundle on $Z$. 
\end{lem}

\begin{prop}\label{f-prop6.5}  
Proposition \ref{f-prop1.9} holds even when $K_X$ is replaced with  
$K_X\otimes E$, where 
$E$ is a Nakano semipositive vector bundle on $X$. 
\end{prop}

\begin{proof}
By Lemma \ref{f-lem6.3} and Lemma \ref{f-lem6.4}, 
the proof of Proposition \ref{f-prop1.9} 
in Section \ref{x-sec4} works for $K_X\otimes E$. 
\end{proof}

Therefore, by Proposition \ref{f-prop6.5} and 
the proof of Theorem \ref{f-thm1.4} and 
Corollary \ref{f-cor1.7} in Section \ref{x-sec4}, 
it is sufficient to prove the following 
theorem for Theorem \ref{f-thm1.12}. 

\begin{thm}[Theorem \ref{f-thmA} twisted 
by Nakano semipositive vector bundles]\label{f-thm6.6}
Let $E$ be a Nakano semipositive 
vector bundle on a compact K\"ahler manifold $X$. 
Let $F$ $($resp.~$M$$)$ 
be a line bundle on a compact K\"ahler manifold $X$ 
with a singular Hermitian 
metric $h$ $($resp.~a smooth Hermitian metric $h_M$$)$ satisfying 
\begin{align*}
\sqrt{-1}\Theta_{h_M} (M) \geq 0 \text{ and }
\sqrt{-1}\Theta_{h}(F) - b \sqrt{-1}\Theta_{h_M} (M) \geq 0 
\text{ for some $b>0$}. 
\end{align*}
Then for a $($nonzero$)$ section $s \in H^{0}(X, M)$, 
the multiplication map induced by $\otimes s$ 
\begin{equation*}
\times s\colon H^{q}(X, K_X\otimes E \otimes F \otimes \mathcal J(h)) 
\xrightarrow{\quad \otimes s \quad } 
H^{q}(X, K_X\otimes E \otimes F \otimes \mathcal J(h) \otimes M )
\end{equation*}
is injective for every $q$. 
Here $K_X$ is the canonical bundle of $X$ 
and $\mathcal J(h)$ is the multiplier ideal sheaf of $h$. 
\end{thm}

We will explain how to modify the proof of Theorem \ref{f-thm5.1} 
for Theorem \ref{f-thm6.6}. 
 
\begin{proof} 
We replace $(F, h_{\e})$ with 
$(E \otimes F, h_{E} h_{\e} )$ in the proof of Theorem \ref{f-thm5.1}, 
where $\{h_{\e} \}_{1\gg \e>0}$ is a family of 
singular Hermitian metrics on $F$ (constructed in Step \ref{f-st1}) 
and $h_{E}$ is a smooth Hermitian metric on $E$ such that 
$\sqrt{-1}\Theta_{h_{E}}(E)$ is Nakano semipositive. 
Then it is easy to see that 
essentially the same proof as in Theorem \ref{f-thm5.1} 
works for Theorem \ref{f-thm6.6} 
thanks to the assumption on the curvature of $E$. 
For the reader's convenience, 
we give several remarks on 
the differences with the proof of Theorem \ref{f-thm5.1}. 

There is no problem 
when we construct $h_{\e}$ and $\omega_{\e,\delta}$. 
In Step \ref{f-st4} in the proof 
of Theorem \ref{f-thm5.1}, we used the de Rham--Weil isomorphism 
(see (\ref{eq5.7}) and \cite[Proposition 5.5]{matsumura4}), 
which was constructed by using Lemma \ref{f-lem5.12}. 
Since \cite[4.1 Th\'eor\`eme]{demailly-dbar} (which yields Lemma \ref{f-lem5.12}) 
is formulated for holomorphic vector bundles, 
Lemma \ref{f-lem5.12} can be generalized to $(E \otimes F, h_{E} h_{\e} )$.
From this generalization, 
we can construct the de Rham--Weil isomorphism for $E \otimes F$ 
\begin{align*}
\overline{f_{\e,\delta}} \colon \dfrac{\Ker \dbar}{\Image \dbar} 
\text{ of } L^{n,q}_{(2)}(E \otimes F)_{h_{E} h_{\e}, \omega_{\e,\delta}}  
\xrightarrow{\quad \cong \quad }
\dfrac{\Ker \mu}{\Image \mu} 
\text{ of } C^{q}(\mathcal{U}, K_X\otimes E \otimes F \otimes \I{h_{\e}}).   
\end{align*}

In Step \ref{f-st1}, we used the orthogonal 
decomposition of $L^{n,q}_{(2)}(F)_{\e, \delta}$, 
which was obtained from the fact 
that $\Image\dbar \subset L^{n,q}_{(2)}(F)_{\e, \delta}$ is closed. 
To obtain the same conclusion 
for $L^{n,q}_{(2)}(E \otimes F)_{h_{E} h_{\e}, \omega_{\e,\delta}}$, 
it is sufficient to show that 
$C^{q}(\mathcal{U}, K_{X}\otimes E \otimes F \otimes \I{h_{\e}})$ 
is a Fr\'echet space (see \cite[Proposition 5.8]{matsumura4}). 
We can easily check it by 
using the same argument as in \cite[Theorem 5.3]{matsumura4} 
for $\mathbb{C}^{{\rm{rank}} E}$-valued holomorphic functions. 

The argument of Step \ref{f-st2} 
works even if we consider  $(E \otimes F, h_{E} h_{\e} )$. 
In Step \ref{f-st3}, we need to prove \eqref{eq5.6}, 
but it is easy to see 
\begin{align*}
-\e q |u_{\e,\delta}|^{2}_{h_{E} h_{\e}, \omega_{\e,\delta}}
&\leq
\langle \sqrt{-1}\Theta_{h_{\e}}(F)
\Lambda_{\omega_{\e, \delta}} u_{\e,\delta}, u_{\e,\delta}
\rangle_{h_{E} h_{\e}, \omega_{\e,\delta}} \\
&\leq 
\langle \sqrt{-1}\Theta_{h_{E} h_{\e} }(E \otimes F)
\Lambda_{\omega_{\e, \delta}} u_{\e,\delta}, u_{\e,\delta}
\rangle_{h_{E} h_{\e}, \omega_{\e,\delta}}  
\end{align*}
since $\sqrt{-1}\Theta_{h_{E}}(E)$ is Nakano semipositive.  
\end{proof}

When $E$ is Nakano semipositive and is not flat, 
there seems to be no Hodge theoretic approach 
to Theorem \ref{f-thm6.6} even if $h$ is smooth. 


\end{document}